\documentclass[11pt]{article}
\usepackage{amsfonts,amssymb,amsmath,amscd,latexsym,makeidx,amsthm, color,graphicx,multicol,pdfsync}
\usepackage[utf8]{inputenc}
\usepackage[english]{babel}
\usepackage{comment}
\newtheorem{teo}{Theorem}[section]
\newtheorem{lemma}{Lemma}[section]
\newtheorem{cor}{Corollary}[section]
\newtheorem{oss}{Remark}[section]
\newtheorem{prop}{Proposition}[section]

\newcommand{\ra}{\longrightarrow}
\newcommand{\dis}{\displaystyle}
\newcommand{\bs}{\backslash}

\newcommand{\ov}{\overline}

\newcommand{\R}{\mathbb{R}}
\newcommand{\ti}{\widetilde}

\newcommand{\N}{\mathbb{N}}

\newcommand{\eps}{\varepsilon}

\newcommand{\ph}{\varphi}
\newcommand{\D}[2]{ \frac{\partial #1}{ \partial #2}}

\newcommand{\rw}{\rightharpoonup}

\newenvironment{Si}[1]{\left\{\begin{array}{#1}}{\end{array} \right. }

\title{Extremal Functions for Singular Moser-Trudinger Embeddings}
\author{Stefano Iula\thanks{The authors are supported by the Swiss National Science Foundation project nr. PP00P2-144669.}\\ {\small Universit\"at Basel}\\ {\small \texttt{stefano.iula@unibas.ch} }\and Gabriele Mancini${}^*$ \thanks{The author is supported by SISSA and the PRIN project \emph{Variational and perturbative aspects of nonlinear differential problems.}}\\  {\small Universit\"at Basel}\\  {\small \texttt{gabriele.mancini@unibas.ch} }}
\date{}
\begin{document}
\maketitle

\begin{abstract}
We study Moser-Trudinger type functionals in the presence of singular potentials. In particular we propose a proof of a singular Carleson-Chang type estimate by means of Onofri's inequality for the unit disk in $\R^2$. Moreover we extend the analysis of \cite{AdDr} and \cite{Csa} considering Adimurthi-Druet type functionals on compact surfaces with conical singularities and discussing the existence of extremals for such functionals.
\end{abstract}

\section{Introduction}

Let $\Omega\subseteq \R^2$ be a bounded  domain, from the well known Sobolev's inequality 

\begin{equation}\label{Sobolev}
\|u\|_{L^\frac{2p}{2-p}(\Omega)} \le S_p \|\nabla u\| _{L^p(\Omega)}  \qquad  p \in ( 1,2) ,\; u\in W^{1,p}_0(\Omega),
\end{equation}

one can deduce that the Sobolev space $H^1_0(\Omega):=W^{1,2}_0(\Omega)$ is embedded into $L^q(\Omega)$ $\forall\; q\ge 1$. A much more precise result was proved in 1967 by Trudinger \cite{Trud}: on bounded subsets of $H^1_0(\Omega)$ one has  uniform exponential-type integrability. Specifically, there exists $\beta>0$ such that
\begin{equation}\label{Tru}
\sup_{u\in H^1_0(\Omega),\;\int_{\Omega}|\nabla u|^2 dx \le 1}\int_{\Omega} e^{\beta u^2} dx <+\infty.
\end{equation}
This inequality was later improved by Moser in \cite{mos}, who proved that the sharp exponent in \eqref{Tru} is $\beta = 4\pi$, that is 
\begin{equation}\label{MT1}
\sup_{u\in H^1_0(\Omega),\;\int_{\Omega}|\nabla u|^2 dx \le 1}\int_{\Omega} e^{4\pi u^2} dx <+\infty,
\end{equation}
and 
\begin{equation}\label{MT2}
\sup_{u\in H^1_0(\Omega),\;\int_{\Omega}|\nabla u|^2 dx \le 1}\int_{\Omega} e^{\beta u^2} dx =+\infty
\end{equation}
for $\beta >4\pi$. 
An interesting question consists in studying the existence of extremal functions for \eqref{MT1}. Indeed, while there is no function realizing equality in \eqref{Sobolev}, one can prove that the supremum in \eqref{MT1} is always attained. This was  proved in \cite{CC} by Carleson and Chang for the unit disk  $D\subseteq \R^2$, and by Flucher (\cite{Flu}) for  arbitrary bounded domains  (see also \cite{stru} and \cite{Lin}). 
The proof of these results is based on a concentration-compactness  alternative stated by P. L. Lions (\cite{Lions}): for a sequence $u_n\in H^1_0(\Omega)$ such that $\|\nabla u_n\|_{L^2(\Omega)}=1$  one has, up to subsequences, either 
$$
\int_{\Omega} e^{4\pi u_n^2} dx \rightarrow \int_{\Omega} e^{4\pi u^2} dx
$$
where $u $ is the weak limit of $u_n$, or $u_n$ concentrates in a point  $x\in \ov{\Omega}$, that is
\begin{equation}\label{con}
|\nabla u|^2dx \rw \delta_x \qquad \mbox{ and }\qquad u_n\rw 0.
\end{equation}
The key step in \cite{CC} consists in proving that if a  sequence of radially symmetric functions $u_n\in H^1_0(D)$  concentrates at 0,  then
\begin{equation}\label{CC}
\limsup_{n\to \infty} \int_{D} e^{4\pi u_n^2} dx\le \pi(1+ e).
\end{equation}
Since for the unit disk the supremum in  \eqref{MT1} is strictly greater than $\pi(1+e)$, one can exclude concentration for maximizing sequences by means of \eqref{CC} and therefore prove existence of extremal functions for \eqref{MT1}. In \cite{Flu} Flucher observed that concentration at arbitrary points of a general domain $\Omega$ can always  be reduced, through properly defined rearrangements, to concentration of radially symmetric functions on the unit disk. In particular he proved that if $u_n\in H^1_0(\Omega)$ satisfies $\| \nabla u_n\|_2=1$ and 
\eqref{con}, then
\begin{equation}\label{Flu}
\limsup_{n\to \infty} \int_{\Omega} e^{4\pi u_n^2} dx\le \pi e^{1+4\pi A_\Omega(x)} + |\Omega|.
\end{equation}
where $A_\Omega(x)$ is the Robin function of $\Omega$, that is the trace of the regular part of the Green function of $\Omega$. He also proved 
\begin{equation}
\sup_{u\in H^1_0(\Omega),\int_{\Omega}|\nabla u|^2 dx\le 1} \int_{\Omega} e^{ 4\pi u^2} dx   >\pi e^{1+4\pi \max_{\ov{\Omega}}A_\Omega}+|\Omega|, 
\end{equation}
which implies the existence of extremals for \eqref{MT1} on $\Omega$. Similar results hold if $\Omega$ is replaced by a smooth closed surface $(\Sigma,g)$.  Let us denote
$$
\mathcal{H} := \left\{ u\in H^1(\Sigma)\;:\; \int_{\Sigma}|\nabla u|^2 dv_g \le 1,\; \int_{\Sigma} u \; dv_g =0\right\}.
$$
Fontana \cite{fontana} proved that 
\begin{equation}\label{MT3}
\sup_{u\in \mathcal{H}}\int_{\Sigma} e^{4\pi u^2} dv_g <+\infty
\end{equation}
and 
\begin{equation}\label{MT4}
\sup_{u\in \mathcal{H}}\int_{\Sigma} e^{\beta u^2} dv_g =+\infty
\end{equation}
$\forall\; \beta>4\pi$. Existence of extremal functions for \eqref{MT3} was proved in \cite{Li1} (see also \cite{Li3}, \cite{Li2}), again by excluding concentration for maximizing sequences.

In this paper we are interested in  Moser-Trudinger type inequalities in the presence of singular potentials.  The simplest example is given by the singular metric $|x|^{2\alpha} |dx|^2$ on a bounded domain $\Omega \subset \R^2$ containing the origin. In  \cite{adim} Adimurthi and Sandeep observed that $\forall\; \alpha\in (-1,0]$,
\begin{equation}\label{Adi}
\sup_{u\in H^1_0(\Omega),\int_{\Omega}|\nabla u|^2 dx\le 1}\int_{\Omega} |x|^{2\alpha}  e^{4\pi (1+\alpha) u^2} dx<+\infty,
\end{equation}
and 
\begin{equation}\label{Adi2}
\sup_{u\in H^1_0(\Omega),\int_{\Omega}|\nabla u|^2 dx\le 1}\int_{\Omega} |x|^{2\alpha}  e^{\beta u^2} dx=+\infty, 
\end{equation}
for any $\beta >4\pi(1+\alpha)$. Existence of extremals for \eqref{Adi} has recently been proved in \cite{Csa2} and \cite{Csa}. The strategy is similar to the one used for the case $\alpha=0$. One can exclude concentration for maximizing sequences using the following estimate, which can be obtained from \eqref{CC} using  a simple change of variables (see  \cite{adim}, \cite{Csa}).
\begin{teo}\label{teo 1}
Let $u_n\in H^1_0(D)$ be such that $\int_{D}|\nabla u_n|^2 dx\le 1$ and $u_n \rw 0$ in $H^1_0(D)$, then $\forall \; \alpha \in (-1,0]$ we have
\begin{equation}\label{CC gen}
\limsup_{n\to \infty} \int_{D} |x|^{2\alpha} e^{4\pi (1+\alpha)u_n^2} dx \le \frac{\pi (1+e)}{1+\alpha}.
\end{equation}
\end{teo}
In the first part of this work we will give a simplified version of the argument in \cite{CC} and show that \eqref{CC} (and therefore \eqref{CC gen}) can be deduced from  Onofri's inequality for the unit disk.
\begin{prop}[See \cite{onofri}, \cite{Beck}]
\label{OnofrisuD}
For any $u\in H^1_0(D)$ we have
\begin{equation}\label{Onofri Disco}
\log\left(\frac{1}{\pi}\int_{D}  e^{u}dx\right) \le \frac{1}{16\pi}\int_{D} |\nabla u|^2 dx +1.
\end{equation}
\end{prop}

\medskip

Theorem \ref{teo 1} can be used to prove existence of extremals for several generalized versions of \eqref{MT1}. In  \cite{AdDr}  Adimurthi and Druet  proved that
\begin{equation}\label{AdDr}
\sup_{u\in H^1_0(\Omega),\int_{\Omega}|\nabla u|^2 dx \le 1} \int_{\Omega}  e^{4\pi u^2(1+\lambda\|u\|_{L^2(\Omega)}^2)} dx <+\infty
\end{equation}
for any $\lambda<\lambda(\Omega)$, where $\lambda(\Omega)$ is the first eigenvalue of $-\Delta$ with respect to Dirichlet boundary conditions.   This bound on $\lambda$ is sharp, that is
\begin{equation}\label{AdDr2}
\sup_{u\in H^1_0(\Omega),\int_{\Omega}|\nabla u|^2 dx \le 1} \int_{\Omega}  e^{4\pi u^2(1+\lambda(\Omega)\|u\|_{L^2(\Omega)}^2)} dx =\infty. 
\end{equation}
Existence of extremal functions for sufficiently small $\lambda$ for this improved inequality has been proved in \cite{LuYang} and \cite{Yang-C}. Similar results hold for compact surfaces on the space $\mathcal{H}$.
We refer to \cite{Tint}, \cite{Yang} and references therein for further improved inequalities. 

\medskip

In this work we will focus on Adimurthi-Druet type inequalities on  compact surfaces with conical singularities. Given a smooth, closed Riemannian surface $(\Sigma,g)$, and a finite number of  points $p_1,\ldots,p_m\in \Sigma$  we will consider  functionals of the form
\begin{equation}\label{E}
E^{\beta,\lambda,q}_{\Sigma,h}(u):= \int_{\Sigma} h e^{ \beta u^2 (1+\lambda\|u\|_{L^q(\Sigma,g)}^2)} dv_g
\end{equation}
where $\lambda,\beta\ge 0$, $q> 1$ and $h\in C^1(\Sigma \bs \{p_1,\ldots,p_m\})$ is a positive function satisfying
\begin{equation}\label{hh}
h(x)\approx d(x,p_i)^{2\alpha_i} \qquad  \mbox{ with }\quad \alpha_i>-1 \; \mbox{ near }\; p_i   \quad i=1,\ldots,m.
\end{equation}
One of the main  motivations for the choice of these singular weights comes indeed from the study of surfaces with conical singularities.  We recall that a smooth metric $\ov{g}$ on  $\Sigma\bs \{p_1,\ldots,p_m\}$ is said to have conical singularities of order $\alpha_1,\ldots,\alpha_m$ in $p_1,\ldots,p_m$ if $\ov{g}=h g$ with $g$ smooth metric on $\Sigma$ and $0<h\in C^{\infty}(\Sigma\bs\{p_1,\ldots,p_m\})$ satisfying  \eqref{hh}. Thus the functional \eqref{E} naturally appears in the analysis of Moser-Trudinger embeddings for the singular surface $(\Sigma, \ov g)$ (see \cite{troyanov}).  

If $m=0$ and $h\equiv 1$, $E_{\Sigma,1}^{\beta,\lambda ,q}$ corresponds to the functional studied in \cite{LuYang}. In particular, one has
\begin{equation}\label{LuY}
\sup_{u\in \mathcal{H}}  E^{4\pi,\lambda,q}_{\Sigma,1} < +\infty \quad \Longleftrightarrow \quad  \lambda<\lambda_{q}(\Sigma,g),
\end{equation}
where 
$$
\lambda_q(\Sigma,g):= \inf_{u\in \mathcal{H}} \frac{\int_{\Sigma}|\nabla u|^2dv_g}{\|u\|_{L^q(\Sigma,g)}^2}.
$$  
As it happens for \eqref{Adi}, if $h$ has singularities the critical exponent   becomes smaller.  More precisely, in \cite{troyanov} Troyanov (see also \cite{ChenMT}) proved that if $h$ is a positive function satisfying \eqref{hh}, then 
\begin{equation}\label{MT Troy}
\sup_{u\in \mathcal{H}} E_{\Sigma,h}^{\beta,0,q} <+\infty \qquad \Longleftrightarrow \qquad \beta\le 4\pi(1+\ov{\alpha}) 
\end{equation}
where  $\dis{\ov{\alpha}= \min\left\{0,\min_{1\le i\le m}\alpha_i\right\}}$. 
In this paper we combine \eqref{LuY} and \eqref{MT Troy}  obtaining the following singular version of \eqref{LuY}.

\begin{teo}\label{teo 3}
Let $(\Sigma,g)$ be a smooth, closed, surface. If $h\in C^1(\Sigma\bs\{p_1,\ldots,p_m\})$ is a positive function satisfying \eqref{hh}, then  $\forall\; \beta \in [0,4\pi(1+\ov \alpha)]\;$ and $\;  \lambda\in [0,\lambda_q(\Sigma,g))$ we have 
\begin{equation}\label{SUP}
\sup_{u\in \mathcal{H}} E_{\Sigma,h}^{\beta,\lambda,q}(u)<+\infty,
\end{equation}
and the supremum is attained if $\beta<4\pi(1+\ov{\alpha})$ or if $\beta = 4\pi (1+\ov \alpha)$ and $\lambda$ is sufficiently small. 
Moreover 
$$
\sup_{u\in \mathcal{H}} E_{\Sigma,h}^{\beta,\lambda,q}(u)=+\infty
$$
for $\beta> 4\pi (1+\ov{\alpha})$, or $\beta=4\pi(1+\ov{\alpha})$ and $\lambda>\lambda_q(\Sigma,g)$.
\end{teo}

Note that we do not treat the case $\beta= 4\pi (1+\ov{\alpha})$ and $\lambda= \lambda_q(\Sigma,g)$ (see Remark \ref{caso limite}).  In Theorem \ref{teo 3}, it is possible to replace $\|\cdot\|_{L^q(\Sigma,g)}$, $\lambda_q(\Sigma,g)$ and  $\mathcal{H}$, with  $\|\cdot\|_{L^q(\Sigma,g_h)}$,  $\lambda_q(\Sigma,g_h)$ and
$$
\mathcal{H}_h:=\left\{u\in H^1_0(\Sigma)\;:\; \int_{\Sigma}|\nabla_{g_h} u|^2 dv_{g_h}\leq 1,\; \int_{\Sigma} u \; dv_{g_h} =0\right\},
$$ 
where $g_h:=h g$. In particular we can extend the Adimurthi-Druet  inequality to  compact surfaces with conical singularities.

\begin{teo}\label{teo 4}
Let $(\Sigma,g)$ be a closed surface with conical singularities of order $\alpha_1,\ldots,\alpha_m>-1$ in $p_1,\ldots,p_m\in \Sigma$. Then for any  $0\le \lambda < \lambda_q(\Sigma,g)$ we have 
$$
\sup_{u\in \mathcal{H}} \int_{\Sigma} e^{4\pi(1+\ov{\alpha})u^2 (1+\lambda \|u\|^2_{L^q(\Sigma,g)} ) } dv_{g}<+\infty,
$$
and the supremum is attained for $\beta<4\pi(1+\ov{\alpha})$ or  for $\beta = 4\pi (1+\ov\alpha)$ and  sufficiently small $\lambda$. Moreover 
$$
\sup_{u\in \mathcal{H}} \int_{\Sigma} e^{\beta u^2 (1+\lambda \|u\|^2_{L^q(\Sigma,g)} ) } dv_{g}=+\infty,
$$
if $\beta>4\pi (1+\ov{\alpha})$ or $\beta= 4\pi (1+\ov{\alpha})$ and $\lambda>\lambda_q(\Sigma,g)$. 
\end{teo}

As in \cite{Li1}, \cite{Yang} and \cite{LuYang}, our techniques can be adapted  to treat the case of compact surfaces with boundary.

\begin{teo}\label{teo5} Let $(\Sigma,g)$ be a smooth compact Riemannian surface with boundary. If $p_1,\ldots,p_m\in \Sigma \bs \partial\Sigma$ and $h\in C^1(\Sigma\bs\{p_1,\ldots,p_m\})$ satisfies \eqref{hh},  then $\forall\; \beta \in [0,4\pi(1+\ov{\alpha})]\;$ and $\;  \lambda\in [0,\lambda_q(\Sigma,g))$ we have 
$$
\sup_{u\in H^1_0(\Sigma), \int_{\Sigma}|\nabla u|^2 dv_g\le 1} E_{\Sigma,h}^{\beta,\lambda,q}(u)<+\infty 
$$
and the supremum is attained if $\beta<4\pi(1+\ov{\alpha})$ or if $\beta = 4\pi (1+\ov\alpha)$ and $\lambda$ is sufficiently small. Furthermore if $\beta> 4\pi (1+\ov{\alpha})$, or $\beta=4\pi(1+\ov{\alpha})$ and $\lambda\ge \lambda_q(\Sigma,g)$, we have 
$$
\sup_{u\in u\in H^1_0(\Sigma), \int_{\Sigma}|\nabla u|^2 dv_g\le 1} E_{\Sigma,h}^{\beta,\lambda,q}(u)=+\infty.
$$
\end{teo}

In particular, if $\Sigma = \ov{\Omega}$ is the closure of a    bounded domain in $\R^2$, Theorem \ref{teo5} gives the following generalization of the results in \cite{Flu}, \cite{AdDr}, \cite{Csa}. 

\begin{cor}
Let $\Omega \subseteq \R^2$ be a bounded domain. For any choice of $V\in C^1(\ov{\Omega})$, $V>0$, $\alpha_1,\ldots,\alpha_m>-1$,  $\; x_1,\ldots,x_m\in \Omega$, $q>1$ and $\lambda\in [0,\lambda_q(\Omega))$, the  supremum 
$$
\sup_{u\in H^{1}_0(\Omega),\; \int_{\Omega}|\nabla u|^2dx\le 1} \int_{\Omega}V(x) \prod_{i=1}^m|x-x_i|^{2\alpha_i}  e^{4\pi (1+\ov{\alpha}) u^2\left(1+\lambda\|u\|_{L^q(\Omega)}^2\right)} dx 
$$
is finite. Moreover it is attained if  $\lambda$ is sufficiently small. 
\end{cor}

This paper is organized as follows. Section \ref{sez3} contains a simple proof of  Theorem \ref{teo 1}.  Theorem \ref{teo 3} will be proved  in the remaining three sections. In section \ref{sez4} we will state some useful lemmas and prove existence of extremals for $E^{\beta, \lambda,q}_{\Sigma,h}$ with $\beta<4\pi(1+\ov{\alpha})$. In Section \ref{sez5} we will deal with the blow-up analysis for maximizing sequences for the critical case $\beta=4\pi(1+\ov{\alpha})$ and  we will prove an estimate similar to \eqref{Flu}, which implies the finiteness of the supremum in \eqref{SUP}.  Finally, in Section \ref{sez6} we test the functionals with a properly defined family of functions and complete the proof Theorem \ref{teo 3}. In the Appendix we will discuss some Onofri-type inequalities. In particular we will show how to deduce \eqref{Onofri Disco} from the standard Onofri inquality on $S^2$ and discuss its extensions to singular disks.  The proof of Theorems \ref{teo 4} and \ref{teo5} is very similar to the one of Theorem \ref{teo 3}, hence it will not be discussed in this work.


\section{A Carleson-Chang type estimate.}\label{sez3}
In this section we will prove Theorem \ref{teo 1} by means of \eqref{Onofri Disco}. We will consider the space
$$
H:=\left\{u\in H^1_0(D)\;:\; \int_{D}|\nabla u|^2 dx\le 1\right\}
$$
and, for any $\alpha\in (-1,0]$, the functional 
$$
E_\alpha(u):= \int_{D} |x|^{2\alpha} e^{4\pi (1+\alpha) u^2 } dx.
$$
By \eqref{Adi} we have $\sup_{H} E_\alpha<+\infty$. For any $\delta>0$, we will denote with $D_\delta$ the disk with radius $\delta$ centered at 0.

\begin{oss}\label{scala}
 With a trivial change of variables, one immediately gets that if $\delta>0$ and  $u\in H^1_0(D_\delta)$ are such that $\int_{D_\delta}|\nabla u_n|^2 dx \le 1$, then
$$
\int_{D_\delta} |x|^{2\alpha} e^{4\pi (1+\alpha)u^2} dx  \le \delta^{2(1+\alpha)} \sup_{H} E_\alpha.
$$
\end{oss}

In order to control the values of the Moser-Trudinger functional on a small scale, we will need the following scaled version of \eqref{Onofri Disco}   (cfr. Lemma 1 in \cite{CC}). 

\begin{cor}\label{cor CC}
For any $\delta,\tau >0$ and $c\in \R$ we have 
$$
 \int_{D_\delta}  e^{c u} \, dx  \le  \pi e^{1+\frac{c^2 \tau}{16\pi}} \delta^{2}
$$
for any $u\in H^1_0(D_\delta)$ such that $\int_{D_\delta} |\nabla u|^2\, dx\le \tau$.
\end{cor}


As in the original proof in \cite{CC}, we will first assume $\alpha=0$ and work with radially symmetric functions. For this reason we introduce the spaces
$$
H^{1}_{0,rad}(D):=\left\{u\in H^1_0(D)\;:\; u \mbox { is radially symmetric and decreasing}\right\}.
$$
and
$$
H_{rad}:=H \cap H^1_{0,rad}(D).
$$ 

\medskip

Functions in $H_{rad}$ satisfy the following useful decay estimate.
\begin{lemma}\label{decay}
For any $u\in H_{rad}$ we have 
$$
u(x)^2\le -\frac{1}{2\pi}\left(1-\int_{D_{|x|}}|\nabla u|^2 dy\right) \log| x| \qquad \forall\; x\in D\bs\{0\}.
$$
\end{lemma}
\begin{proof}
We bound
\[
\begin{split}
|u(x)| &\le  \int_{|x|}^1 |u'(t)| dt \le \left(\int_{|x|}^1 t u'(t)^2 dt  \right)^\frac{1}{2} \left(-\log |x|\right)^\frac{1}{2}\\& 
\le\frac{1}{\sqrt{2\pi}} \left(\int_{D\bs D_{|x|}} |\nabla u|^2 dy\right)^\frac{1}{2}\left(-\log |x|\right)^\frac{1}{2} \\&
\le \frac{1}{\sqrt{2\pi}} \left(1-\int_{D_{|x|}} |\nabla u|^2 dy\right)^\frac{1}{2}\left(-\log |x|\right)^\frac{1}{2}.
\end{split}
\]
\end{proof}

On a sufficiently small scale, it is possible to control $E_0$ using only Corollary \ref{cor CC} Lemma \ref{decay} and Remark \ref{scala}.

\begin{lemma}\label{disketto}
If $u_n\in H_{rad}$ and $\delta_n\ra 0$  satisfy  
\begin{equation}
\int_{D_{\delta_n}}|\nabla u_n|^2 dx \ra 0,
\end{equation}
then 
$$
\limsup_{n\to \infty}
\int_{D_{\delta_n}} e^{4\pi u_n^2} dx \le  \pi e.
$$
\end{lemma}
\begin{proof}
Take $v_n:= u_n - u_n(\delta_n) \in H^1_0(D_{\delta_n})$ and set  $\tau_n:=\int_{D_{\delta_n}}|\nabla u_n|^2 dx$. 
If $\tau_n=0$, then $u_n\equiv u_n(\delta_n)$ in $D_{\delta_n}$ and, using Lemma \ref{decay}, we find
$$
\int_{D_{\delta_n}} e^{4\pi u_n^2 } dx =\pi  \delta_n^{2}  e^{4\pi u_n(\delta_n)^2} \le \pi < \pi e.
$$
Thus, w.l.o.g. we can assume $\tau_n >0$ for every $n \in \N$.  By Holder's inequality  and Remark \ref{scala} we have
\begin{equation}\label{bfds}
\begin{split}
&\int_{D_{\delta_n}} e^{4\pi u_n^2 } dx = e^{4\pi u_n(\delta_n)^2} \int_{D_{\delta_n}} e^{4\pi v_n^2 + 8\pi  u_n(\delta_n) v_n}  dx  \\&  
\le e^{4\pi u_n(\delta_n)^2 } \left( \int_{D_{\delta_n}}e^{4\pi\frac{ v_n^2}{\tau_n}} dx \right)^{\tau_n} \left(\int_{D_{\delta_n}}  e^{\frac{8\pi u_n(\delta_n) v_n}{1-\tau_n}} dx \right)^{1-\tau_n} \\&
\le e^{4\pi u_n(\delta_n)^2}\left( \delta_n^{2} \sup_H E_0 \right)^{\tau_n} \left(\int_{D_{\delta_n}}  e^{\frac{8\pi u_n(\delta_n) v_n}{1-\tau_n}} dx \right)^{1-\tau_n}.
\end{split}
\end{equation}
Applying Corollary \ref{cor CC} with $\tau=\tau_n, \delta=\delta_n$ and $c=\frac{8\pi u_n(\delta_n)}{1-\tau_n}$ we find 
$$
\int_{D_{\delta_n}}  e^{\frac{8\pi u_n(\delta_n) v_n}{1-\tau_n}} dx \le \delta_n^{2} {\pi e^{1+\frac{4\pi u_n(\delta_n)^2 }{(1-\tau_n)^2}\tau_n}}
$$
thus from \eqref{bfds} it follows
\[
\begin{split}
\int_{D_{\delta_n}} e^{4\pi  u_n^2} dx&\le  \delta_n^2 \left( \sup_H E_0 \right)^{\tau_n} \left(\pi e \right)^{1-\tau_n}e^{4\pi u_n^2(\delta_n) + \frac{4\pi u_n(\delta_n)^2\tau_n}{(1-\tau_n)}}\\&
=\delta_n^{2} \left( \sup_H E_0 \right)^{\tau_n} \left(\pi e \right)^{1-\tau_n}e^{\frac{4\pi u_n(\delta)^2}{1-\tau_n} }.
\end{split}
\]
Lemma \ref{decay} yields 
$$
\delta_n^{2} e^{4\pi \frac{u_n(\delta_n)^2}{1-\tau_n}} \le 1,
$$
therefore
$$
\int_{D_{\delta_n}} e^{4\pi  u_n^2} dx \le   \left( \sup_H E_0 \right)^{\tau_n} \left(\pi e\right)^{1-\tau_n}.
$$
Since $\tau_n\ra 0$, we obtain the conclusion by taking the $\limsup$ as $n\to \infty$ on both sides. 
\end{proof}

In order to prove Theorem \ref{teo 1} on $H_{rad}$ for $\alpha=0$, it is sufficient to show that, if $u_n\rw 0$, there exists a sequence $\delta_n$ satisfying the hypotheses of Lemma  \ref{disketto} and such that
\begin{equation}\label{con0}
\int_{D_{\delta_n}}  \left( e^{4\pi u_n^2} -1\right) dx \ra 0.
\end{equation}
Note that, by dominated convergence theorem, \eqref{con0} holds  if there exists $f\in  L^1(D)$ such that  
\begin{equation}\label{giust}
 e^{4\pi u_n^2} \le f
\end{equation}  in $D\bs D_{\delta_n}$. In the next lemma we  will chose a function  $f\in L^1(D)$ with critical growth near $0$ (i.e. $f(x)\approx \frac{1}{|x|^2\log^2|x|}$) and define $\delta_n$ so that \eqref{giust} is satisfied.

\begin{lemma}\label{raggio CC} 
Take $u_n\in H_{rad}$ such that  
\begin{equation}\label{conv unif}
\sup_{D\bs D_r} u_n \ra 0 \qquad \forall\; r\in (0,1).
\end{equation}
Then there exists a sequence $\delta_n \in (0,1)$ such that 
\begin{enumerate}
\item \label{pr1} $\delta_n\ra 0$.
\item \label{pr2} $\tau_n:=\int_{D_{\delta_n}}|\nabla u_n|^2 dx \ra 0$.
\item \label{pr3} $\int_{D\bs D_{\delta_n}} e^{4\pi u_n^2} dx \ra \pi  $.
\end{enumerate}
\end{lemma}
\begin{proof}
We consider the function
\begin{equation}\label{fa}
f(x):=\begin{Si}{cl}
\frac{1}{|x|^2 \log^2|x|}  & \; |x|\le e^{-1}
\vspace{0.2cm}\\
 e^2 & \; |x| \in (e^{-1},1].
\end{Si}
\end{equation} 
Note that $f\in L^1(D)$ and 
\begin{equation}\label{inf f}
\inf_{ (0,1)} f = e^2.
\end{equation}
Let us fix $\gamma_n\in (0,\frac{1}{n})$ such that $\int_{D_{\gamma_n}}|\nabla u_n|^2 dx \le \frac{1}{n}.$
We define 
$$
\ti{\delta}_n := \inf\left\{r \in (0,1)\;:\; e^{4\pi u_n^2(x)} \le f(x) \;  \mbox{ for } r \le |x| \le 1  \right\}\in [0,1).
$$
and
$$
\delta_n := \begin{Si}{cc}
\ti{\delta}_n  & \mbox{ if }\ti{\delta}_n >0 \\
\gamma_n  & \mbox{ if }\ti{\delta}_n =0.
\end{Si}
$$
By definition we have 
$$
e^{4\pi  u_n^2 } \le f \qquad \mbox{ in }  \quad D\bs D_{\delta_n},
$$
thus \emph{\ref{pr3}} follows by dominated convergence Theorem.  To conclude the proof it suffices to prove that if $n_k\nearrow +\infty$ is chosen so that  $\delta_{n_k}= \ti{\delta}_{n_k}$ $\forall\; k$, then
\begin{equation}\label{fine}
\lim_{k\to \infty}\delta_{n_k}= \lim_{k\to \infty} \tau_{n_k}=0.
\end{equation} 
For such $n_k$ one has
\begin{equation}\label{bosad}
e^{4\pi u_{n_k}(\delta_{n_k})^2}=f(\delta_{n_k}).
\end{equation}
In particular using \eqref{inf f} we obtain
$$
e^{4\pi u_{n_k}(\delta_{n_k})^2} =f(\delta_{n_k}) \ge e^2>1
$$
which, by \eqref{conv unif}, yields $\delta_{n_k}\stackrel{k\to \infty}{\ra}0$.
Finally, Lemma \ref{decay} and \eqref{bosad} imply
$$
1\ge \delta_{n_k}^{2(1-\tau_{n_k})} e^{4\pi u_{n_k}(\delta_{n_k})^2}= \frac{\delta_{n_k}^{-2\tau_{n_k}}}{ \log^2\delta_{n_k}}
$$
so that $\tau_{n_k}\stackrel{k\to \infty}{\ra}0$ (otherwise the limit of the RHS would be $+\infty$). 
\end{proof}

Combining Lemma \ref{disketto} with Lemma \ref{raggio CC} we immediately get \eqref{CC} for radially symmetric functions: 

\begin{prop}\label{CC radial}
Let $u_n\in H_{rad}$ and $\alpha\in (-1,+\infty]$. If for any $r\in (0,1)$
\[
\sup_{D\bs D_r} u_n \ra 0,
\]
then 
\[
\limsup_{n\to \infty}E_\alpha (u_n) \le \frac{\pi (1+e)}{(1+\alpha)}.
\]
\end{prop}

\begin{proof}




If $\alpha=0$ the proof follows directly applying Lemma \ref{raggio CC} and Lemma \ref{disketto}. 
If $\alpha\neq 0$ consider
$$v_n(x)=(1+\alpha)^{\frac 12}u_n(|x|^{\frac{1}{1+\alpha}}).$$ 

We have
$$
\int_{D}|\nabla v_n|^2\,  dx = \int_{D}|\nabla u_n|^2 \, dx
$$
and hence $v_n\in H_{rad}$. Moreover we compute
$$
\int_{D} |x|^{2\alpha} e^{(1+\alpha)u_n^2} \,dx=\frac{1}{1+\alpha}\int_{D} e^{4\pi v_n^2} \, dx
$$

and the claim follows at once from the case $\alpha=0$.

\end{proof}

To pass from Proposition \ref{CC radial} to Theorem \ref{teo 1} we will use symmetric rearrangements.  
We recall that given a measurable function $u: \R^2 \ra [0,+\infty)$, the symmetric decreasing rearrangement of $u$ is the unique right-continuous radially symmetric and decreasing  function $u^*:\R^2 \ra [0,+\infty) $ such that 
$$
|\{u>t\}| = |\{u^*>t\}| \qquad \forall\; t>0.
$$
Among the properties of $u^*$ we recall that
\begin{enumerate}
\item If $u \in L^p(\R^2)$, then $u^*\in L^p(\R^2)$ and $\|u_*\|_p=\|u\|_p$.
\item If $u\in H^1_0(D)$, then $u^*\in H^1_{0,rad}(D)$ and 
\begin{equation}\label{poly}
\int_{D}|\nabla u^*|^2 dx \le \int_{D}|\nabla u|^2dx.
\end{equation}
\item If $u,v:\R^2\ra [0,+\infty)$, then  
\begin{equation}\label{dec2}
\int_{\R^2} u^*(x) v^*(x) dx \ge \int_{\R^2} u(x) v(x) dx.
\end{equation}
In particular if $u\in H^1_0(D)$ and $\alpha \le 0$, 
\begin{equation}\label{dec}
\int_{D} |x|^{2\alpha} e^{u^*} dx \ge \int_{D} |x|^{2\alpha} e^{u} dx. 
\end{equation}
\end{enumerate}
Note that \eqref{dec} does not hold if $\alpha>0$. We refer the reader to \cite{kesa} for a more detailed introduction to symmetric rearrangements.

\begin{proof}[Proof of Theorem \ref{teo 1}]
Take $u_n\in H$ such that $u_n\rw 0$ and let $u_n^*$ be the symmetric decreasing rearrangement of $u_n$.  Then $u_n^*\in H_{rad}$ and since $\|u_n^*\|_2= \|u_n\|_2\ra 0$ we have $\sup_{D\bs D_r} u_n^*\ra 0$ $\forall\; r>0$.  Thus from \eqref{dec} and  Proposition  \ref{CC radial} we get 
$$
\limsup_{n\to \infty} E_\alpha(u_n) \le \limsup_{n\to \infty}E_\alpha(u_n^*) \le \frac{\pi (1+e)}{1+\alpha}.
$$
\end{proof}

In the next section we will need the following local version of Theorem \ref{teo 1}.  

\begin{cor}\label{mi serve}
Fix  $\delta>0$, $\alpha\in(-1, 0]$ and take $u_n\in H^1_0(D_\delta)$ such that $\int_{D_\delta} |\nabla u_n|^2 dx \le 1$ and  $u_n\rw 0$ in $H^1_0(D_\delta)$. 
For any choice of sequences $\delta_n\to 0$, $x_n\in \Omega$ such that $D_{\delta_n}(x_n) \subset D_\delta$ we have 
$$
\limsup_{n\to \infty} \int_{D_{\delta_n}(x_n)} |x|^{2\alpha}e^{4\pi (1+\alpha)u_n^2} dx \le \frac{\pi e}{1+\alpha}\delta^{2(1+\alpha)}.
$$
\end{cor}
\begin{proof}
Let us define $\ti{u}_n(x):= u_n({\delta}x).$ Note that $\ti{u}_n \in H$ and satisfies the hypotheses  of Theorem \ref{teo 1}, thus 
\[
\begin{split}
\limsup_{n\to \infty} &\int_{D_\delta} |x|^{2\alpha} (e^{4\pi u_n^2} -1)\, dx \\&=  \delta^{2(1+\alpha)} \limsup_{n\to \infty} \int_{D} |x|^{2\alpha} (e^{4\pi \ti{u}_n^2} -1)\, dx\\&
\le \delta^{2(1+\alpha)} \frac{\pi e}{1+\alpha}.
\end{split}
\]
Thus we get
\[
\begin{split}
\limsup_{n\to \infty} &\int_{D_{\delta_n}(x_n)} |x|^{2\alpha} e^{4\pi (1+\alpha) u_n^2} dx\\&
= \limsup_{n\to \infty} \int_{D_{\delta_n}(x_n)} |x|^{2\alpha} \left( e^{4\pi (1+\alpha) u_n^2} -1\right)dx\\&
\le \int_{D_\delta} |x|^{2\alpha} (e^{4\pi u_n^2} -1)dx \\&
\le \delta^{2(1+\alpha)} \frac{\pi e}{1+\alpha}.
\end{split}
\]
\end{proof}

\begin{oss}
We remark that for $\alpha\in(-1,0]$ by Theorem \ref{teo 1} it is enough to show that 

$$\sup_{H} E_\alpha >\frac{\pi (1+e)}{1+\alpha}$$
 in order to prove existence of extremal functions for $E_\alpha$ (see \cite{CC}, \cite{Csa2}).
 \end{oss}

We conclude this section pointing out that as we just did for the Carleson-Chang type estimates, one can have a singular version of the Onofri inequality \eqref{Onofri Disco} (see Proposition \ref{prop2prop3} in the Appendix). In particular one can deduce the following generalized version of Corollary \ref{cor CC}. 

\begin{cor}\label{cor CC sing}
$\forall\; $ $\delta,\tau >0$ , $c\in \R$ and $\alpha\in (-1,0]$ we have 
$$
 \int_{D_\delta}  |x|^{2\alpha}e^{c u} \, dx  \le \frac{ \pi e^{1+\frac{c^2 \tau}{16\pi(1+\alpha)}} \delta^{2(1+\alpha)}}{1+\alpha}
$$
$\forall\; u\in H^1_0(D_\delta)$ such that $\int_{D_\delta} |\nabla u|^2 \, dx\le \tau$.
\end{cor}

%
%


\section{Extremals on Compact Surfaces: Notations and Prelimiaries}\label{sez4}
Let $(\Sigma,g)$ be a smooth, closed Riemannian surface. In this section, and in the rest of the paper, we will fix $p_1,\ldots,p_m\in \Sigma$ and consider a positive function $h\in C^1(\Sigma\bs \{p_1,\ldots,p_m\})$ satisfying \eqref{hh}.  More precisely, denoting by $d$ the Riemannian distance on $(\Sigma, g)$ and by $B_r$ the corresponding metric ball, we will assume that for some $ \delta>0 $,
\begin{equation}\label{hhh}
\frac{h}{d(\;\cdot\;,p_i)^{2\alpha_i}} \in C^1_+(B_\delta(p_i)):=\left\{ f\in C^1(B_\delta(p_i)) \;:\;f>0\right\}\quad \mbox{ for } i=1,\ldots,m.
\end{equation}
In order to distinguish the singular points $p_1,\ldots,p_m$ from the regular ones, we introduce a singularity index function
\begin{equation}\label{ind}
\alpha(x):=\begin{Si}{cl}
\alpha_i & \mbox{ if } x=p_i\\
0 & x\in\Sigma\bs\{p_1,\ldots,p_m\} .
\end{Si}
\end{equation}
Clearly condition \eqref{hhh} implies that the limit
\begin{equation}\label{KK}
K(p):= \lim_{q\to p} \frac{h(q)}{d(q,p)^{2\alpha(p)}}
\end{equation}
exists and is strictly positive for any $p\in \Sigma$. We will study functionals of the form \eqref{E} on the space
$$
\mathcal{H}:=\left\{u\in H^1(\Sigma) \;:\; \int_{\Sigma}|\nabla u|^2 dv_g \le 1,\; \int_{\Sigma} u\; dv_{g}=0\right\}.
$$
To simplify the notation we will set
$$
\ov{\alpha}:=\min\left\{0,\min_{1\le i\le m} \alpha_i\right\}
$$
and
$$
\ov{\beta}:=4\pi(1+\ov{\alpha}).
$$
Given $s\ge 1$, the  symbols $\|\cdot\|_s$, $L^s(\Sigma)$ will denote the standard $L^s-$norm and $L^s-$space on $\Sigma$ with respect to the metric $g$.  Since we will deal with the singular metric $g_h = gh$ we will also consider
$$
\|u\|_{s,h}:=\int_{\Sigma} |u|^s dv_{g_h}=\int_{\Sigma} h \;|u|^s dv_{g} 
$$
and 
$$
L^s(\Sigma,g_h):=\{ u:\Sigma\ra \R\; \mbox{ Borel-measurable}, \; \|u\|_{s,h}<+\infty\}.
$$

In this section we will prove the existence of an extremal function for $E^{\beta,\lambda,q}_{\Sigma,h}$ for the subcritical case $\beta<\ov{\beta}$. We begin by stating some well known but useful Lemmas:

\begin{lemma}\label{integr}
If $u\in H^{1}(\Sigma)$ then $e^{u^2} \in L^s(\Sigma)\cap L^s(\Sigma,g_h)$, $\forall\; s\ge 1$.
\begin{proof}
From \eqref{hhh} we have $h\in L^r(\Sigma)$ for some $r>1$, hence it is sufficient to prove that  $e^{u^2} \in L^s(\Sigma)$, $\forall\; s\ge 1$. Moreover, since
$$
e^{s u^2} = e^{s (u-\ov{u})^2 +2 s (u- \ov{u})\ov{u}+ \ov{u}^2} \le   e^{2 s (u-\ov{u})^2 } e^{2 s \ov{u}^2},
$$
without loss of generality we can assume $\ov{u}=0$. 
 Take $\eps>0$ such that $2 s\eps\le 4\pi $ and a function $v\in C^1(\Sigma)$ satisfying $\|\nabla_g(v-u)\|^2_2\le \eps$  and $\int_{\Sigma} v \;dv_g=0$. By \eqref{MT3}, we have 
\begin{equation}\label{somma norme}
\|e^{2s (u-v)^2}\|_{1} + \|e^{2s\eps \frac{u^2}{\|\nabla u\|_2}}\|_{1}<+\infty.
\end{equation}
Note that
\begin{equation}\label{prod}
e^{s u^2} \le  e^{s(u-v)^2} e^{2s  u v}. 
\end{equation}
By \eqref{somma norme}, we have $e^{s(u-v)^2} \in L^2(\Sigma)$ and, since $v\in L^\infty(\Sigma)$,
$$
e^{2 s u v} \le e^{s \eps \frac{u^2}{\|\nabla u\|_2^2}} e^{C(\eps, s,\|\nabla u\|_2) v^2} \in L^2(\Sigma).
$$
Hence using \eqref{prod} and Holder's inequality we get $e^{s u^2}\in L^1(\Sigma)$. 
\end{proof}
\end{lemma}

\begin{lemma}\label{Lions}   
If $u_n\in \mathcal{H}$ and $u_n\rw u \neq 0$ weakly in $H^1(\Sigma)$,  then 
$$
\sup_n \int_{\Sigma} h e^{p\ov{ \beta} u^2_n }  dv_{g} <+\infty
$$
$\forall\; 1\le p<\frac{1}{1-\|\nabla u\|_{2}^2}$.
\end{lemma}
\begin{proof} Observe that
\begin{equation}\label{bleah}
e^{p\ov{\beta}u_n^2}\le e^{p\ov{\beta} (u_n-u)^2}  e^{2 p\ov{\beta} u_n u}. 
\end{equation}
Since 
$$
\frac{1}{p}>1-\|\nabla u\|_2^2\ge \|\nabla u_n\|^2_2-\|\nabla u\|_2^2 = \|\nabla (u_n-u)\|_2^2 +o(1) \quad \Longrightarrow \quad \limsup_{n\to \infty} \|\nabla (u_n-u)\|_2^2 <\frac{1}{p},
$$
by  \eqref{MT Troy} we get $\|e^{p\ov{\beta}(u_n-u)^2}\|_{s,h}\le C$ for some $s>1$. Taking $\frac{1}{s}+\frac{1}{s'}=1$ and using Lemma \ref{integr} we have
$$ e^{2 p s' \ov{\beta }u_n u} \le e^{\frac{\ov{\beta}}{2} u_n^2} \;  e^{C_{s,\alpha,p} u^2} \in L^1(\Sigma,g_h)\qquad  \Longrightarrow \qquad \|e^{2p \ov{\beta} u_n u}\|_{s',h} \le C.$$
Thus from \eqref{bleah} we get $\|e^{p\ov{\beta}u_n^2}\|_{1,h}\le C$.
\end{proof}

Existence of extremals for $\beta<\ov{\beta}$ is a simple consequence of Lemma \ref{Lions} and Vitali's convergence Theorem. 

\begin{lemma}\label{subcr}
$\forall\; \beta \in (0,\ov{\beta})$, $\lambda\in [0,\lambda_q(\Sigma,g))$, $q> 1$ we have 
$$
\sup_{\mathcal{H}} E_{\Sigma,h}^{\beta,\lambda,q} <+\infty
$$
and  the supremum is attained. 
\end{lemma}
\begin{proof}
Let $u_n\in \mathcal{H}$ be a maximizing sequence for $E_{\Sigma,h}^{\beta,\lambda,q}$, and assume $u_n \rw u$ weakly in $H^1(\Sigma)$. We claim that $e^{\beta u_n^2(1+\lambda\|u_n\|_q^2)}$ is uniformly bounded in $L^p(\Sigma,g_h)$ for some $p>1$. In particular by Vitali's convergence theorem we get $E^{\beta,\lambda,q}_{\Sigma,h}(u_n)\ra E^{\beta,\lambda,q}_{\Sigma,h}(u)$ with $E^{\beta,\lambda,q}_{\Sigma,h}(u)<+\infty$.  Hence 
$
E^{\beta,\lambda,q}_{\Sigma,h}(u) = \sup_{\mathcal{H}} E^{\beta,\lambda,q}_{\Sigma,h}(u),
$
proving the conclusion. 

If $u=0$, then
$$
\beta (1+\lambda \|u_n\|_q^2)\ra \beta <\ov{\beta},
$$
and the claim is proved taking $1<p<\frac{\ov{\beta}}{\beta}$ and using \eqref{MT Troy}. If $u\neq 0$, since 
$$
(1-\|\nabla u\|_2^2)(1+\lambda\|u_n\|_q^2) \le 1-\|\nabla u\|_2^2+\lambda\|u\|_q^2 +o(1)\le  1-(\lambda_q(\Sigma)-\lambda)\|u\|_q^2 +o(1)<1
$$
we can find $p>1$  such that $\dis{\limsup_{n\to \infty}p(1+\lambda\|u_n\|_q^2) <\frac{1}{1-\|\nabla u\|_2^2}}$, and the claim follows from Lemma \ref{Lions}.
\end{proof}

The behaviour of extremal functions as $\beta\to\ov\beta$ will be studied in the next section. As for now we can study the convergence of the suprema.
\begin{lemma}\label{conv sup} As $\beta \nearrow \ov{\beta}$ we have
$$
 \sup_{\mathcal{H}} E_{\Sigma,h}^{\beta,\lambda,q} \ra \sup_{\mathcal{H}} E_{\Sigma,h}^{\ov{\beta},\lambda,q}.
$$
\end{lemma}
\begin{proof}
Clearly, since $\beta<\ov{\beta}$, we have
$$
\limsup_{\beta \nearrow \ov{\beta}}\; \sup_{\mathcal{H}} E_{\Sigma,h}^{\beta,\lambda,q}\le  \sup_{\mathcal{H}} E_{\Sigma,h}^{\ov{\beta},\lambda,q}.
$$
On the other hand, by monotone convergence theorem we have 
$$
\liminf_{\beta \nearrow \ov{\beta}} \sup_{\mathcal{H}} E_{\Sigma,h}^{\beta,\lambda,q}\ge \liminf_{\beta \nearrow \ov\beta} E_{\Sigma,h}^{\beta,\lambda,q}(v)=   E_{\Sigma,h}^{\ov{\beta},\lambda,q}(v) \qquad\forall\; v\in \mathcal{H},
$$
which gives 
$$
\liminf_{\beta \nearrow \ov{\beta}}  \sup_{\mathcal{H}} E_{\Sigma,h}^{\beta,\lambda,q}\ge   \sup_{\mathcal{H}} E_{\Sigma,h}^{\ov{\beta},\lambda,q}.
$$
\end{proof}

We conclude this section with some Remarks concerning isothermal coordinates and  Green's functions.  We recall that, given any point $p\in \Sigma$, we can always find a small neighborhood $\Omega$ of $p$ and a local chart 
\begin{equation}\label{psi1}
\psi: \Omega \ra D_{\delta_0}\subset\R^2
\end{equation} 
such that 
\begin{equation}
\psi (p)=0
\end{equation}
and 
\begin{equation}\label{pullb}
(\psi^{-1})^* g = e^{\ph} |dx|^2
\end{equation}
where \begin{equation}
\ph\in C^\infty(\ov{D_{\delta_0}}) \qquad \mbox{and}\qquad\ph(0)=0.
\end{equation} For any $\delta<\delta_0 \;$ we will denote $\Omega_\delta:=\psi^{-1} (D_{\delta})$.  More generally if $D_r(x)\subseteq D_{\delta_0}$ we define $\Omega_r(\psi^{-1}(x)):= \psi^{-1}(D_r(x))$.   We stress that \eqref{KK} and \eqref{pullb} also imply 
\begin{equation}
(\ph^{-1})^* g_h = |x|^{2\alpha(p)} V(x)e^{\ph} |dx|^2.
\end{equation}
with \begin{equation}\label{psi2}
0<V\in C^0(\ov{D_{\delta_0}})\qquad \mbox{and}\qquad V(0)=K(p).
\end{equation}

For any $p\in \Sigma$ we denote $G_p^\lambda$ the solution of 
\begin{equation}\label{defGR}
\begin{Si}{l}
\dis{-\Delta_g G^\lambda_p= \delta_p + \lambda \|G^\lambda_p\|_q^{2-q}|G^\lambda_p|^{q-2} G^\lambda_p- \frac{1}{|\Sigma|} \left(1+ \lambda \|G^\lambda_p\|_q^{2-q}\int_{\Sigma} |G^\lambda_p|^{q-2}G^\lambda_p dv_g \right)}\\
\dis{\int_{\Sigma} G^\lambda_p dv_g =0}.
\end{Si}
\end{equation}
In local coordinates satisfying \eqref{psi1}-\eqref{psi2} we have 
\begin{equation}\label{exp Gr}
G_p^\lambda (\psi^{-1}(x)) = -\frac{1}{2\pi } \log |x| +A_p^\lambda +\xi(x)
\end{equation}
with $\xi \in C^1(\ov{D_{\delta_0}})$ and $\xi(x)=O(|x|)$. Observe that $G_p^0$ is the standard Green's function for $-\Delta_g$.

\begin{lemma}\label{as0}
As $\lambda \to 0$ we have $G^\lambda_p \ra G^0_p$ in  $L^s(\Sigma)$ $\forall\;s$ and $A^\lambda_p \ra A_p^0$. 
\end{lemma}
\begin{proof}
Let us denote $\dis{c_\lambda:= \frac{\lambda}{|\Sigma|} \|G^\lambda_p\|_q^{2-q}\int_{\Sigma} |G^\lambda_p|^{q-2}G^\lambda_p dv_g}$. Observe that 
$$
-\Delta_g (G^\lambda_p-G_p^0)=\lambda \|G^\lambda_p\|_q^{2-q}|G^\lambda_p|^{q-2} G^\lambda_p-c_\lambda.$$
Since
$$
\left\| \|G^\lambda_p\|_q^{2-q}|G^\lambda_p|^{q-2} G^\lambda_p\right\|_{\frac{q}{q-1}} = \|G_p^\lambda\|_q 
$$
by elliptic estimates we find 
\begin{equation}\label{elli}
\|G_p^\lambda-G^0_p\|_{\infty}\le \|G_p^\lambda-G^0_p\|_{W^{2,\frac{q}{q-1}}(\Sigma)} \le C \lambda \|G_p^\lambda\|_q.
\end{equation}
In particular
$$
\|G_p^\lambda\|_q \le \|G^0_p\|_q +\|G_p^\lambda-G^0_p\|_q \le\|G_p^0\|_q +  C \|G_p^\lambda-G^0_p\|_\infty \le \|G_p^0\|_q + C \lambda   \|G_p^\lambda\|_q,
$$
hence for sufficiently small $\lambda$ we have 
$$
\|G_p^\lambda\|_q \le C \|G_p^0\|_q. 
$$
Thus by \eqref{elli}, as $\lambda \to 0$ we find 
$$\|G_p^\lambda-G_p^0\|_\infty \ra 0.$$
In particular $G_p^\lambda\ra G_p^0$ in $L^s$ for any $s>1$.  Since 
$ A^\lambda_p - A^0_p = (G_p^\lambda - G^0_p)(p)$ we also get $A^\lambda_p \to A_p^0$.
\end{proof}

\begin{lemma}\label{Grad Gr} Fix $p\in\Sigma$ and let $(\Omega,\psi)$ be a local chart satisfying \eqref{psi1}-\eqref{psi2}. As $\delta\to 0$ we have 
$$
\int_{\Sigma \bs \Omega_\delta } |\nabla G_p^\lambda|^2 dv_g = -\frac{1}{2\pi} \log \delta +A^\lambda_p +\lambda\| G_p^\lambda\|^2_q   + O(\delta|\log\delta|).
$$
\end{lemma}
\begin{proof}
Integrating by parts we have 
\begin{equation}\label{G1}
\int_{\Sigma\bs \Omega_\delta} |\nabla G_p^\lambda|^2 dv_g = -\int_{\Sigma\bs \Omega_\delta}\Delta_g G_p^\lambda \; G_p^\lambda dv_g -\int_{\partial\Omega_\delta} G_p^\lambda \D{G_p^\lambda}{\nu} d\sigma_g.
\end{equation}
For the first term, using the definition of $G_p^\lambda$  we get
\begin{equation}\label{G2}
\begin{split}
-\int_{\Sigma\bs \Omega_\delta}\Delta_g G_p^\lambda \; G_p^\lambda dv_g& = \lambda\|G_p^\lambda\|_q^{2-q} \int_{\Sigma\bs \Omega_{\delta}} |G_p^\lambda|^q dv_g   -\left(\frac{1}{|\Sigma|}+c_\lambda\right)\int_{\Sigma \bs \Omega_\delta} G_p^\lambda\; dv_g \\&
= \lambda \|G_p^\lambda\|_q^2 +o(1).
\end{split}
\end{equation} 
For the second term we use \eqref{exp Gr} to find 

\begin{equation}\label{G3}
-\int_{\partial\Omega_\delta} G_p^\lambda \D{G_p^\lambda}{\nu} d\sigma_g = -\frac{1}{2\pi} \log \delta +A^\lambda_p +O(\delta |\log \delta|). 
\end{equation}

\end{proof}

\section{Blow-up Analysis for the Critical Exponent.}\label{sez5}
In this section we will study the critical case $\beta = \ov{\beta}$.

Let us fix $q> 1, \lambda\in [0,\lambda_q(\Sigma,g))$ and take a  sequence $\beta_n \nearrow \ov{\beta}$ ( $\beta_n <\ov{\beta}$ for any $n\in \mathbb{N}$). To simplify the notation we will set $E_n:= E^{\beta_n,\lambda,q}_{\Sigma,h}$. By Lemma \ref{subcr}, for any $n$ we can take a function $u_n\in \mathcal{H}$ such that 
\begin{equation}\label{sup sub}
E_n(u_n)=\sup_{\mathcal{H}} E_n.
\end{equation}
Up to subsequences, we can always assume that 
\begin{equation}\label{wlim}
u_n \rw u_0 \qquad \mbox{ in } H^1(\Sigma) 
\end{equation}
and
\begin{equation}
\hspace{1cm}u_n \ra u_0 \quad \mbox{ in } L^s(\Sigma) \quad \forall\; s\ge 1. 
\end{equation}

\begin{lemma}\label{non0}
If $u_0 \neq 0$, then
\begin{equation}\label{goal}
E_n(u_n) \ra E_{\Sigma,h}^{\ov{\beta},\lambda,q}(u_0)<+\infty.
\end{equation} In particular
$$\sup_{\mathcal{H}}E_{\Sigma,h}^{\ov{\beta},\lambda,q}<+\infty$$
and $u_0$ is an extremal function. 
\end{lemma}
\begin{proof}
If $u_0\neq 0$ we can argue as in Lemma \ref{subcr} to find $p>1$ such that   $e^{\beta_n u_n^2(1+\lambda \|u_n\|_q^2)}$ is uniformly bounded in $L^p(\Sigma,g_h)$.  
Vitali's convergence Theorem yields \eqref{goal}. 
Lemma \ref{conv sup} implies 
$$
\sup_{\mathcal{H}} E_{\Sigma,h}^{\ov{\beta},\lambda,q}=  E_{\Sigma,h}^{\ov{\beta},\lambda,q}(u_0)<+\infty. 
$$
\end{proof}

Thus it is sufficient to study the case $u_0=0$, which we will assume for the rest of this section. In the same spirit of Theorem \ref{teo 1} and \eqref{Flu}, we will prove the following sharp upper bound for $E_n(u_n)$. 


\begin{prop}\label{finale}
If $u_0=0$ we have
$$
\limsup_{n\to \infty} E_n(u_n)\le  \frac{\pi e}{1+\ov\alpha}\max_{p\in\Sigma,\, \alpha(p)=\ov\alpha}K(p) e^{\ov{\beta} A^\lambda_p} +|\Sigma|_{g_h}
$$
where $A^\lambda_p$ is defined as in \eqref{exp Gr} and $|\Sigma|_{g_h}:= \int_{\Sigma} h\; dv_g$.
\end{prop}

\begin{oss}\label{ossmax}
We remark that the quantity
$$
\max_{p\in\Sigma,\, \alpha(p)=\ov\alpha}K(p) e^{\ov{\beta} A^\lambda_p}
$$
is well defined. Indeed if $\ov\alpha<0$ the set of points such that $\alpha(p)=\ov\alpha$ is finite. On the other hand if $\ov\alpha=0$ we have that $K\equiv h$ on $\Sigma\setminus\left\{p_1,\dots,p_m\right\}=\left\{p\in\Sigma\colon \alpha(p)=\ov\alpha\right\}$ and the function $he^{\ov\beta A_p^\lambda}$ is continuous on $\Sigma$ and has zeros at the points $p_1,\dots, p_m$.

\medskip

In particular Lemma \ref{non0} and Proposition \ref{finale} give a proof of an Adimurthi-Druet type inequality, namely
$$\sup_{\mathcal{H}}E_{\Sigma,h}^{\ov{\beta},\lambda,q}<+\infty.$$
\end{oss}

The rest of this section is devoted to the proof of Proposition \ref{finale}. 

\medskip


\begin{lemma}\label{equazione}
There exists $s>1$ such that  $u_n\in \mathcal{H} \cap W^{2,s}(\Sigma)$ $\forall\; n$. Moreover $\|\nabla u_n\|_2=1$ and we have 
\begin{equation}\label{equn}
-\Delta_g u_n = \gamma_n h(x) u_n  e^{b_n u_n^2} +s_n(x)
\end{equation}
where
\begin{equation}\label{b_n}
b_n:= \beta_n(1+\lambda\|u_n\|_q^2)\ra \ov{\beta},
\end{equation}
\begin{equation}\label{gamma_n}
\limsup_{n}\gamma_n<+\infty\quad\; \mbox{and} \quad
\gamma_n \int_{\Sigma}  h\;u_n^2 e^{u_n^2} dv_{g} \ra 1,
\end{equation}
and
\begin{equation}\label{sn2}
s_n:= \lambda_n \|u_n\|_{q}^{2-q} |u_n|^{q-2} u_n - c_n
\end{equation}
with
\begin{equation} 
\lambda_n \ra \lambda,
\end{equation}
and
\begin{equation}\label{c_n}
 c_n := \frac{1}{|\Sigma|}\left(\gamma_n\int_{\Sigma}  u_n e^{b_n u^2_n} dv_{g_h} +\lambda_n\|u_n\|_q^{2-q}\int_{\Sigma} |u_n|^{q-2}u_n dv_g\right).
\end{equation}
In particular we have
\begin{equation}\label{sn1}
 c_n\ra 0,\qquad\|s_n\|_{\frac{q}{q-1}} \ra 0
\end{equation}
as $n\to +\infty$.
\end{lemma}
\begin{proof}
The maximality of $u_n$ clearly implies $\|\nabla u_n\|_2=1$. 
Using Langrange multipliers theorem, it is simple to verify that $u_n$ satisfies
\begin{equation}\label{primaeq}
-\Delta_g u_n =  \nu_n b_n h(x) u_n e^{b_n u_n^2} + \lambda \nu_n \beta_n \mu_n \|u_n\|_q^{2-q}|u_n|^{q-2}u_n  -c_n.
\end{equation}
where $b_n$ is defined as in \eqref{b_n},  $\mu_n:= \int_{\Sigma}h\; u_n^2 e^{b_n u_n^2} dv_{g}$,
\begin{equation}\label{c_n2}
 c_n := \frac{1}{|\Sigma|}\left(\gamma_n\int_{\Sigma}  h u_n e^{b_n u^2_n} dv_{g} +\lambda \nu_n \beta_n \mu_n \|u_n\|_q^{2-q}\int_{\Sigma} |u_n|^{q-2}u_n dv_g\right),
\end{equation} and $\nu_n \in  \R$. 
We define $\gamma_n:=\nu_n b_n$, $\lambda_n:=\lambda \nu_n \beta_n \mu_n$ and $s_n(x) :=\lambda_n \|u_n\|_q^{2-q} |u_n|^{q-2} u_n-c_n$ \\so that \eqref{equn} and \eqref{sn2} are satisfied.
Observe also that
\begin{equation}\label{la1}
\left\|\|u_n\|_q^{2-q} |u_n|^{q-2} u_n\right\|_{\frac{q}{q-1}} = \|u_n\|_q \ra 0. 
\end{equation} and
\begin{equation}\label{la2}
\|u_n\|_q^{2-q}\left|\int_{\Sigma} |u_n|^{q-2}u_n dv_g\right| \le \|u_n\|_{q}|\Sigma|^\frac{1}{q}\ra 0
\end{equation}
If $s_0>1$ is such that $h\in L^{s_0}(\Sigma)$, using Lemma \ref{integr}  and standard Elliptic regularity, we find $u_n \in W^{2,s}(\Sigma)$ $\forall\; 1<s<\min\{s_0,\frac{q}{q-1}\}$.   Multiplying \eqref{primaeq} by $u_n$ and integrating on $\Sigma$  we get
$$
1= \nu_n b_n \mu_n + \lambda \nu_n \beta_n \mu_n \|u_n\|_q^2 = \nu_n b_n \mu_n (1+\frac{\lambda \beta_n\|u_n\|_q^2}{b_n})= \gamma_n \mu_n (1+o(1))
$$
from which we get the second part of \eqref{gamma_n}. As a consequence we also have 
\begin{equation}\label{la3}
\lambda_n= \lambda \nu_n \beta_n \mu_n =\lambda \gamma_n \mu_n \frac{\beta_n}{b_n} \ra \lambda.
\end{equation} 
Now we prove $\dis{\limsup_{n\to \infty}\gamma_n<+\infty}$ or,  equivalently, $\dis{\liminf_{n\to \infty }\mu_n   >0}$.   For any $t>0$, we have 
$$
E_n(u_n) \le \frac{1}{t^2}\int_{\{|u_n|>t\}} h\; u_n^2 e^{b_n u_n^2} dv_{g} +\int_{\{|u_n|\le t\}} h e^{b_n u_n^2} dv_{g} \le \frac{1}{t^2}\int_{\Sigma} h u_n^2 e^{b_n u_n^2} dv_{g} +|\Sigma|_{g_h} +o(1)
$$ from which 
$$
\liminf_{n\to \infty} \mu_n= \liminf_{n\to \infty}\int_{\Sigma} h\; u_n^2 e^{b_n u_n^2} dv_{g}\ge t^2 \left(\sup_{\mathcal{H}}E_{\Sigma,h}^{\ov{\beta},\lambda,q}-|\Sigma|_{g_h}\right)>0.
$$
It remains to prove that $c_n\ra 0$ which, with \eqref{la1} and \eqref{la3}, completes the proof of \eqref{sn1}. For any $t>0$
$$
\gamma_n \int_{\Sigma} h |u_n| e^{b_n u_n^2} dv_{g}\le \frac{\gamma_n}{t}\int_{\{|u_n|>t\}} h u_n^2 e^{b_n u_n^2} dv_{g} + \gamma_n \int_{\{|u_n|\le t\}}h |u_n| e^{b_n u_n^2} dv_{g} = \frac{1+o(1)}{t} +o(1).
$$ 
Since $t$ can be taken arbitrarily large we find 
$$
\gamma_n \int_{\Sigma}h u_n e^{b_n u_n^2} dv_{g}\ra 0. 
$$
Combined with \eqref{c_n}, \eqref{la2} and \eqref{la3}, this yields $c_n\to0$. 
\end{proof}

By Lemma \ref{equazione} we know that $u_n\in C^0(\Sigma)$, thus we can take a sequence $p_n$ such that 
\begin{equation}\label{m e p}
m_n:=\max_{\Sigma} |u_n| = u_n(p_n),
\end{equation}
where the last equality holds up to changing the sign of $u_n$. 
Clearly if $\sup_{n}m_n<+\infty$, then we would have $E_n(u_n)\ra |\Sigma|_{g_h}$ which contradicts Lemma \ref{conv sup}. Thus, up to subsequences, we will assume 
\begin{equation}\label{m e p2}
m_n\ra+\infty\qquad \mbox{ and } \qquad p_n\ra p. 
\end{equation}

\begin{lemma}\label{punto} 
Let $\Omega\subset\Sigma$ be an open subset such that 
$$ \limsup_{n\to+\infty} \| \nabla u_n \|_{L^2(\Omega)}<1.$$
Then $$ \|u_n\|_{L^\infty_{loc}(\Omega)}\le C.$$
\end{lemma}
\begin{proof}
Fix $\tilde\Omega\Subset\Omega$. Take a cut-off function $\xi \in C_0^\infty(\Omega)$ such that $0\le \xi \le 1$ and $\xi \equiv 1$ in $\Omega'$ where $\tilde \Omega\Subset\Omega'\Subset\Omega$.
Since 
$$
\int_{\Sigma} |\nabla u_n \xi|^2 dv_g  = \int_{\Omega}|\nabla u_n|^2 \xi^2 dv_g + 2\int_{\Omega} u_n \xi \nabla u_n \cdot \nabla \xi \;dv_g +\int_{\Omega}|\nabla \xi|^2 u_n^2 dv_g \le 
$$
$$
\le (1+\eps)\int_{\Omega}|\nabla u_n|^2 \xi^2 dv_g + C_\eps \int_{\Omega} |\nabla \xi |^2 u_n^2 dv_g 
$$ 
and $\eps$ can be taken arbitrarily small, we find 
$$
\limsup_{n\to \infty}\|\nabla (u_n \xi)\|^2_{L^2(\Sigma)}< 1. 
$$
Thus, applying \eqref{MT Troy} to $v_n:= \frac{\xi u_n}{\|\nabla(\xi u_n)\|_{L^2(\Sigma)} }$ we find
\begin{equation}\label{soloqui}
\left\|e^{\ov{\beta} u_n^2 (1+\lambda \|u_n\|_q^2)}\right\|_{L^{s_0}(\Omega',g_h)} \le C
\end{equation}
for some $s_0>1$. From \eqref{sn1} and \eqref{soloqui},  $-\Delta_g u_n $ is uniformly bounded in $L^{s}(\Omega')$ for any $s<\min\{s_0,\frac{q}{q-1}\}$. If we take another cut-off function $\ti{\xi}\in C_0^\infty(\Omega') $ such that $\ti{\xi} \equiv 1$ in $\tilde\Omega$, applying elliptic estimates to $\ti{\xi} u_n$ in $\Omega'$ we find $\sup_{\Omega'} \tilde\xi u_n\le C$ and hence $\sup_{\tilde\Omega} u_n\le C$. 
\end{proof}

From Lemma \ref{punto} one can deduce that $|\nabla u_n|^2\rw\delta_p$, that is $u_n$ concentrates at $p$. Intuitively it is natural to expect that concentration for maximizing sequences happens in the regions in which $h$ is larger. 
We will show that $p$ must be a minimum point of the singularity index $\alpha$ defined in \eqref{ind}. This will clarify the difference between the cases $\ov{\alpha}<0$ and $\ov{\alpha}=0$: in the former, the blow-up point $p$ will be one of the singular points $p_1,\ldots,p_m$, while in the latter $p\in \Sigma\bs\{p_1,\ldots,p_m\}$ (cfr. Remark \ref{alfapalfabar} and Proposition \ref{keystep}).
The next step consists in studying the behaviour of $u_n$ around $p$. Arguing as in \cite{Li1} we will prove that a suitable scaling of $u_n$ converges to a solution of a (possibly singular) Liouville-type equation on $\R^2$ (see Proposition \ref{rescaling}).

\medskip

Again we consider a local chart $(\Omega, \psi)$ satisfying \eqref{psi1}-\eqref{psi2}. From now on we will denote $x_n:=\psi(p_n)$ and
\begin{equation}\label{defvenne}
v_n= u_n \circ \psi^{-1}.
\end{equation}

Define $t_n$ and $\tilde t_n$ so that

\begin{equation}\label{tn}
t_n^{2(1+\alpha(p))}\gamma_n m_n^2 e^{b_n m_n^2}=1,
\end{equation}
\begin{equation}\label{ttilden}
\tilde{t}_n^{2}|x_n|^{2\alpha(p)}\gamma_n m_n^2 e^{b_n m_n^2}=1.
\end{equation}

\begin{lemma}\label{tnttilden0}
For any $\beta<\ov\beta$ we have
\[
t_n^{2(1+\alpha(p))}m_n^2 e^{\beta m_n^2}\to 0,\qquad \tilde t^2_n |x_n|^{2\alpha(p)} m_n^2 e^{\beta m_n^2} \to 0
\]
as $n\to +\infty$.
In particular, for any $s\geq 0$ we have
\[
\lim_{n\to+\infty} t_n m_n^s=0,\qquad\lim_{n\to+\infty} \tilde t_n m_n^s=0.
\]
Moreover as $n\to+\infty$ we have
\begin{equation}\label{xntn iff}
\frac{|x_n|}{t_n}\to +\infty \Longleftrightarrow \frac{|x_n|}{\tilde t_n}\to +\infty.
\end{equation}
\end{lemma}

\begin{proof}
Since the result can be proven both for $t_n$ and $\tilde t_n$ with the same argument, we will prove it here only for $t_n$. 
By \eqref{b_n}, \eqref{gamma_n} and \eqref{tn}
\[
\begin{split}
t_n^{2(1+\alpha(p))}m_n^2 e^{\beta m_n^2} &= \frac{e^{(\beta-b_n)m_n^2}}{\gamma_n}
= e^{(\beta-b_n)m_n^2} \int_{\Sigma} h u_n^2 e^{b_n u_n^2}dv_g (1+o(1)) \\&
\le  \int_{\Sigma} h u_n^2 e^{\beta u_n^2}  dv_g (1+o(1)).
\end{split}
\]
Take $s= \frac{\ov{\beta}}{\beta}'$ (i.e.  $1/s +\beta/\ov{\beta}=1$)  and $s_0>1$ such that $h \in L^{s_0}(\Sigma)$. Then 
$$
 \int_{\Sigma} h u_n^2 e^{\beta u_n^2}  dv_g \le \| u_n^2\|_{s,h} \| e^{\ov{\beta} u_n^2 }\|_{1,h}^{\frac{\, \beta \,}{ \ov{\beta}}} \le C \|h\|_{s_0}^\frac{1}{s} \|u_n^2\|_{s s_0'} \ra 0. 
$$
As for the last claim it is enough to observe that from \eqref{tn} and \eqref{ttilden} one computes

$$
\frac{|x_n|}{\tilde t_n}=\left(\frac{|x_n|}{t_n}\right)^{1+\alpha(p)}.
$$
\end{proof}

%
%
%



We define now 

\begin{align}\label{rn}
&r_n:=\begin{cases}
\tilde t_n & \text{ if $\frac{|x_n|}{t_n}\to +\infty$ as $n\to +\infty$,}\\
t_n & \text{ otherwise}
\end{cases}
\end{align}
and
\begin{align}
\label{etaenne}
\eta_n(x):=m_n\left(v_n(x_n+r_n x)-m_n\right)
\end{align}
where the function $\eta_n$ is defined in $D_{\frac{\delta_0}{r_n}}$.

\begin{prop}\label{rescaling}
Up to subsequences, $\eta_n\to \eta_0$ in $C^{0}_{loc}(\R^2)\cap H^1_{loc}(\R^2)$. Moreover
\begin{itemize}
\item[$(i)$]if $\frac{|x_n|}{r_n}\to +\infty$ as $n\to +\infty$ the function $\eta_0$ solves
\begin{align}\label{eqreg}
&-\Delta \eta_0=V(0) e^{2\ov\beta \eta_0}\\&
\label{intreg}
\int_{\R^2}V(0)e^{2\ov\beta\eta_0}\, dx=1;
\end{align}
\item[$(ii)$] if $\frac{|x_n|}{r_n}\to \ov x$ the function $\eta_0$ solves
\begin{align}\label{eqsing}
&-\Delta \eta_0=|x+\ov x|^{2\alpha(p)}V(0) e^{2\ov\beta \eta_0}\\&
\label{intsing}
\int_{\R^2}|x+\ov x|^{2\alpha(p)}V(0)e^{2\ov\beta\eta_0}\, dx=1.
\end{align}
\end{itemize}
\end{prop}

\begin{proof}
If $\frac{|x_n|}{t_n}\to +\infty$ as $n\to +\infty$ then $r_n=\tilde t_n$ and 
 it follows that $\eta_n$ as in \eqref{etaenne} satisfies
$$
-\Delta \eta_n = m_n r_n^2 e^{\ph(x_n+r_n x)} \left( \gamma_n|x_n+r_n x |^{2\alpha(p)}  V(x_n+r_n x)e^{b_n v_n^2} v_n(x_n+r_n x) + s_n(x_n+r_n x)\right)=
$$
$$
=e^{\ph(x_n+r_n x)}\left( \left|\frac{x_n}{|x_n|}+\frac{r_n}{|x_n|} x\right|^{2\alpha(p)} V(x_n+r_nx) \left(1+\frac{\eta_n}{m^2_n}\right) e^{b_n  \left( 2  \eta_n + \frac{\eta_n^2}{m_n^2}\right)} +m_nr_n^2  s_n(x_n +r_n x)\right).
$$

Otherwise we have that $r_n=t_n$ and, up to subsequences, $\frac{|x_n|}{t_n}\to \ov x$ as $n\to+\infty$ and $\eta_n$ satisfies
$$-\Delta  \eta_n =m_n r_n^2e^{\ph(x_n+r_n x)}  \left(  \gamma_n \left|x_n + r_n x\right|^{2\alpha(p)} V(x_n+r_n x)  e^{b_n u_n^2} v_n(x_n+r_n x) +  s_n(x_n +r_n x) \right)=
$$
$$
=  e^{\ph(x_n+r_n y)} \left(\left|\frac{x_n}{r_n} + x\right|^{2\alpha(p)} V(x_n+r_n x)\left(1+\frac{\eta_n}{m^2_n}\right) e^{b_n  \left( 2  \eta_n + \frac{\eta_n^2}{m_n^2}\right)} +m_nr_n^2  s_n(x_n +r_n x)\right).
$$


Observe that from Lemma \ref{tnttilden0} and \eqref{sn1} we have
\begin{equation}\label{sd}
\begin{split}
\int_{D_L} |m_n r_n^2 s_n(x_n +r_n x)|^\frac{q}{q-1}\, dx &= m_n^\frac{q}{q-1} r_n^{\frac{2}{q-1}}\int_{D_{L r_n}(x_n)} \hspace{-0.3cm} |s_n(x)|^\frac{q}{q-1} \, dx \\
&\le  m_n^\frac{q}{q-1} r_n^{\frac{2}{q-1}} \|s_n\|^{\frac{q}{q-1}}_{\frac{q}{q-1}} \to 0
\end{split}
\end{equation}
for any $L>0$. 
Since $ 2  \eta_n + \frac{\eta_n^2}{m_n^2}\leq 0$ and $|\eta_n|\le 2 m^2_n$, for any $L>0$, 
in both cases $(i)$ and $(ii)$ we can find  $s>1$ such that
$$\|-\Delta \eta_n \|_{L^{s}(D_L)} \le C.$$
Moreover $\eta_n(0)=0$, thus we can exploit Sobolev's embeddings Theorems and Harnack's inequality to find a uniform bound for $\eta_n$ in $C^{0,\alpha}(D_{\frac{L}{2}})$ for any $L>0$. Hence with a diagonal argument, we find a subsequence of $\eta_n$ such that $\eta\ra \eta_0$ in $H^1_{loc}(\R^2)\cap C^0_{loc}(\R^2)$. 
Moreover $\eta_0$ is a solution of \eqref{eqreg} or \eqref{eqsing}, depending on our choice of $r_n$. It remains to prove \eqref{intreg} and \eqref{intsing} respectively.
In order to do this we observe that in case $(i)$ we compute 

\begin{equation}\label{contr1}
\begin{split}
1&=-\int_{\Sigma} \Delta_g u_n u_n  d v_g= \gamma_n \int_{\Sigma} h\;  u_n^2 e^{b_n u_n^2} d{v_{g}} +\lambda_n\|u_n\|_q^2 \\&\ge \gamma_n \int_{\Omega_{L r_n}(p_n)} h\; u_n^2 e^{b_n u_n^2} dv_{g} +o(1)
\\& = V(0)\int_{D_L} e^{2\ov{\beta} \eta_0} dx  +o(1).
\end{split}
\end{equation}
In particular it holds (see for instance \cite{clChenLi})
\begin{equation}\label{contoreg}
\lim_{L\to +\infty}V(0)\int_{D_L} e^{2\ov{\beta} \eta_0} dx = \frac{1}{1+\ov{\alpha}} \geq1
\end{equation}
where the last inequality follows from the fact that $\ov\alpha\leq 0$. Hence with \eqref{contr1} we obtain 
\eqref{intreg}. 
Similarly, in case $(ii)$ we have
\begin{equation}\label{contr2}
1=-\int_{\Sigma} \Delta_g u_n u_n  d v_g\geq V(0)\int_{D_L}|x+\ov x|^{2\alpha(p)} e^{2\ov{\beta} \eta_0} dx  +o(1).
\end{equation}
On the other hand (cfr. \cite{praj-tar}) 
\begin{equation}\label{contosing}
\lim_{L\to +\infty}V(0)\int_{D_L} |x+\ov x|^{2\alpha(p)}e^{2\ov{\beta} \eta_0} dx = \frac{1+\alpha(p)}{1+\ov{\alpha}} \geq1
\end{equation}
where now the last inequality follows from the minimality of $\ov\alpha$. Therefore \eqref{intsing} is proven. 
\end{proof}

\begin{oss}\label{alfapalfabar}
From the proof of Proposition \ref{rescaling} it follows that if $\ov\alpha<0$ then by \eqref{contr1} and \eqref{contoreg} we have that only case $(ii)$ is possible. Moreover from \eqref{contr2} and \eqref{contosing} we get $\alpha(p)=\ov\alpha$, that is $p$ must be one of the singular points $p_1,\dots,p_m$.

%
%
%
%

\medskip

We stress that Proposition \ref{rescaling} gives us information on the nature of the point $p$ only in the case $\ov\alpha<0$. To have a deeper understanding of the case $\ov\alpha=0$ and a more complete analysis of the blow-up behaviour of $u_n$ near the point $p$ we will need few more steps (see Proposition \ref{keystep}).
\end{oss}

\begin{lemma}\label{contointegrali}
We have
\begin{itemize}
\item[$(i)$] $\lim_{L\to +\infty} \lim_{n\to +\infty} \int_{\Omega_{L_{r_n}}(p_n)} \gamma_n m_n h u_n e^{b_n u_n^2}\, dv_g=1$
\item[$(ii)$] $\lim_{L\to +\infty} \lim_{n\to +\infty} \int_{\Omega_{L_{r_n}}(p_n)} \gamma_n h u_n^2 e^{b_n u_n^2}\, dv_g=1$
\item[$(iii)$] $\lim_{L\to +\infty} \lim_{n\to +\infty} \int_{\Omega_{L_{r_n}}(p_n)} h e^{b_n u_n^2}\, dv_g=\limsup_{n\to +\infty}\frac{1}{\gamma_n m_n^2}.$
\end{itemize}
\end{lemma}

\begin{proof}
Both $(i)$ and $(ii)$ follow easily from Proposition \ref{rescaling}. We are left with the proof of $(iii)$.

By Proposition \ref{rescaling}, for any $L>0$ we have
\[
\lim_{n\to+\infty} \gamma_n m_n^2 \int_{ \Omega_{L r_n}(p_n)}  h e^{b_n u_n^2} dv_g=1+o_L(1) 
\]
where $o_L(1)\ra 0$ as $L\to \infty$. Hence
$$
\limsup_{n\to \infty}\frac{1}{ \gamma_n m_n^2} = (1+o_L(1))\limsup_{n\to \infty } \int_{ \Omega_{L r_n(p_n)}}  h e^{b_n u_n^2} dv_g
$$
and we can conclude the proof letting $L\to +\infty$. 
\end{proof}

Following \cite{Li1}, for any $A>1$  we define 
\[
u_n^A:=\min\{u_n ,\frac{m_n}{A}\}.
\]

\begin{lemma}\label{letronc}
For any $A>1$ we have 
$$
\limsup_{n\to \infty} \int_{\Sigma}|\nabla u_n^A|^2 dv_g = \frac{1}{A}.
$$
\end{lemma}
\begin{proof}
Integrating by parts we have
\[
\liminf_{n\to \infty}\int_{\Sigma}|\nabla u_n^A|^2 dv_g = \liminf_{n\to \infty}\int_{\Sigma}\nabla u_n^A \cdot \nabla u_n dv_g = \liminf_{n\to+\infty}-\int_{\Sigma} \Delta_g u_n u_n^A\, dv_g.
\] 
Fix now $L>0$. By Proposition \ref{rescaling}, for sufficiently large $n$, $\Omega_{L r_n}(p_n) \subseteq\{u_n >\frac{m_n}{A}\}$, hence using \eqref{equn} and \eqref{sn2} we find
\[
-\int_{\Sigma} \Delta_g u_n \; u_n^A dv_g = \gamma_n \int_{\Sigma} h u_n e^{b_n u_n^2} u_n^Adv_{g} +o(1) \ge  \frac{\gamma_n m_n}{A} \int_{\Omega_{L r_n}(p_n)}  h\; u_n e^{b_n u_n^2} dv_{g} +o(1).
\]
Hence passing to the limit as $n,L\to+\infty$ we obtain 
\begin{equation}\label{troc}
\liminf_{n\to \infty}\int_{\Sigma}|\nabla u_n^A|^2 dv_g = \liminf_{n\to+\infty}-\int_{\Sigma} \Delta_g u_n u_n^A\, dv_g \ge \frac{1}{A}
\end{equation}
where the last inequality follows from Lemma \ref{contointegrali}.
Similarly 
$$
-\int_{\Sigma} \Delta_g u_n \left(u_n -\frac{m_n}{A}\right)^+ dv_g \ge \gamma_n \int_{\Omega_{L r_n}(p_n)} h\; u_n e^{b_n u_n^2}\left( u_n-\frac{m_n}{A}\right) dv_{g}+o(1)$$
we get
\begin{equation}\label{tronc}
\liminf_{n\to \infty} \int_{\Sigma} |\nabla (u_n-\frac{m_n}{A})^+|^2dv_g \ge \frac{A-1}{A},
\end{equation}
again from Lemma \ref{contointegrali}.
Clearly $u_n= u_n^A + (u_n-\frac{m_n}{A})^+$ and $\int_{\Sigma}\nabla u_n^A \cdot \nabla (u_n-\frac{m_n}{A})^+ dv_g=0$ thus $$
1=\int_{\Sigma} |\nabla u_n|^2 dv_g = \int_{\Sigma}|\nabla u_n^A|^2 dv_g +\int_{\Sigma}|\nabla \left(u_n-\frac{m_n}{A}\right)^+|^2 dv_g
$$ 
and from \eqref{troc} and  \eqref{tronc} we find 
$$
\lim_{n\to \infty} \int_{\Sigma}|\nabla u_n^A|^2 dv_g = \frac{1}{A} \quad \mbox{ and }\quad 
\lim_{n\to \infty} \int_{\Sigma}|\nabla \left(u_n-\frac{m_n}{A}\right)^+|^2 dv_g = \frac{A-1}{A}.
$$
\end{proof}

With Lemma \ref{letronc} we have a first rough version of Proposition \ref{finale}.

\begin{lemma}\label{supspezzato}
$$
\limsup_{n\to \infty} E_n(u_n)\le \lim_{L\to+\infty}\lim_{n\to+\infty}\int_{\Omega_{L_{r_n}}(p_n)}he^{b_n u_n^2}\, dv_g + |\Sigma|_{g_h}.
$$
\end{lemma}
\begin{proof}
For any $A>1$ we have 
$$
E_n(u_n) = \int_{\left\{u_n\ge \frac{m_n}{A}\right\}} h e^{b_n u_n^2} dv_g + \int_{\left\{u_n\le  \frac{m_n}{A}\right\}} h e^{b_n (u_n^A)^2} dv_g.
$$
By \eqref{gamma_n},
$$
\int_{\left\{u_n\ge \frac{m_n}{A}\right\}} h e^{b_n u_n^2} dv_g \le \frac{A^2}{m_n^2 }\int_{\Sigma} h u_n^2e^{b_n u_n^2} dv_g =  \frac{A^2}{\gamma_n m_n^2}(1+o(1)).
$$
For the last integral we apply Lemma \ref{letronc}. Since $\limsup_{n\to \infty} \|\nabla u_n^A\|_{2}^2\le \frac{1}{A}<1$, \eqref{MT Troy} implies that $e^{b_n (u_n^A)^2}$ is uniformly bounded in $L^s(\Sigma,g_h)$ for some $s>1$. Thus by Vitali's Theorem 
$$
\int_{\left\{u_n\le  \frac{m_n}{A}\right\}} h e^{b_n (u_n^A)^2} dv_g \le \int_{\Sigma} h e^{b_n (u_n^A)^2} dv_g \ra |\Sigma|_{g_h}.
$$
Therefore we proved 
$$
\limsup_{n\to \infty} E_n(u_n)\le \limsup_{n\to \infty} \frac{A^2}{\gamma_n m_n^2} + |\Sigma|_{g_h}.
$$
As $A \to 1$ we get the conclusion thanks to Lemma \ref{contointegrali}.
\end{proof}

\begin{lemma}\label{mis}
We have
$$\gamma_nm_nh \, u_n e^{b_n u_n^2}\rightharpoonup \delta_p$$
weakly as measures as $n\to +\infty$.
\end{lemma}

\begin{proof}
Take $\xi \in C^0(\Sigma)$. For $L>0$, $A>1$ we have
\[
\begin{split}
&\gamma_n m_n \int_{\Sigma}  h\; u_n e^{b_n u_n^2} \xi dv_{g}\\&  
= \gamma_n m_n \int_{\Omega_{L r_n}(p_n)}  h u_n e^{b_n u_n^2} \xi  dv_{g}\\&  
+  \gamma_n m_n \int_{\{u_n>\frac{m_n}{A}\}\bs \Omega_{L r_n}(p_n)}h u_n  e^{b_n u_n^2} \xi dv_{g}\\&
+  \gamma_n m_n \int_{\{u_n\le\frac{m_n}{A}\}}  h u_n e^{b_n u_n^2} \xi dv_{g}\\&
=:I_n^1+I_n^2+I_n^3.
\end{split}
\] 
We have
\[
I_n^1=\gamma_n m_n\int_{\Omega_{L_{r_n}}(p_n)} h u_n e^{b_n u_n^2}(\xi-\xi(p))\, dv_g+\gamma_n m_n\int_{\Omega_{L_{r_n}}(p_n)} h u_n e^{b_n u_n^2}\xi(p)\, dv_g.
\]
From $\|\xi-\xi(p)\|_{L_\infty(\Omega_{L_{r_n}}(p_n))}\to 0$ as $n\to+\infty$ and Lemma \ref{contointegrali} we have
$$\lim_{L\to \infty}\lim_{n\to \infty} I_n^1 = \xi(p).$$
Similarly, using \eqref{gamma_n}
\[
\begin{split}
| I_n^2| &\leq m_n\int_{\{u_n>\frac{m_n}{A}\}\bs \Omega_{L r_n}(p_n)} \gamma_n h u_n e^{b_n u_n^2} | \xi |dv_{g} \\&
\le A \int_{\{u_n>\frac{m_n}{A}\}\bs \Omega_{L r_n}(p_n)}  \gamma_n h u_n^2 e^{b_n u_n^2} | \xi | dv_{g} \\&
\leq A \|\xi\|_{L^\infty(\Sigma)} \left(1 -\int_{\Omega_{L r_n}(p_n)} \gamma_n h u_n^2 e^{b_n u_n^2}  dv_{g} +o(1)\right).
\end{split}
\]
Therefore, from Lemma \ref{contointegrali},
$$
\lim_{L\to \infty} \lim_{n\to \infty} I_n^2 =0.
$$
For the last integral by Lemma \ref{letronc} and 
\eqref{MT Troy} there exist $s>1, C>0$ such that 
$$
\int_{\Sigma} h e^{s \ov{\beta}(u_n^A)^2} dv_{g}\le C
$$
thus 
$$
|I_n^3 |\le   \gamma_n m_n \|\xi\|_\infty \int_{\Sigma} h | u_n |e^{b_n (u_n)^2} dv_g \le \gamma_n m_n \|\xi\|_\infty \|u_n\|_{s',h} \| e^{\ov{\beta}(u_n)^2}\|_{s,h}= \gamma_n m_n o(1).
$$ 
By $(iii)$ in Lemma \ref{contointegrali} and Lemma \ref{supspezzato} we get that $\gamma_n m_n\ra 0$ and hence we find $|I_n^3| \ra 0 $ which gives the conclusion. 
\end{proof}

Let now $G^\lambda_p$ be the Green's function defined in \eqref{defGR}. Using  Lemma \ref{mis} we obtain:

\begin{lemma}\label{conv Gr}
$m_nu_n \ra G^\lambda_p$ in $C^0_{loc}(\Sigma\bs\{p\})\cap H^1_{loc}(\Sigma\bs\{p\})\cap L^s(\Sigma)$ $\forall\; s>1$.
\end{lemma}
\begin{proof}
First we observe that $\|m_n u_n \|_{q}$ is uniformly bounded. If not we could consider the sequence $w_n := \frac{ u_n}{\|u_n\|_q}$  which satisfies 
$$
-\Delta_g w_n = \gamma_n h \frac{u_n}{ \| u_n\|_q}  e^{b_n u_n^2} + \frac{ s_n}{\| u_n\|_q }
$$
Arguing as in Lemma \ref{mis} one can prove that  $\|\gamma_n h m_n u_n e^{b_n u_n^2}\|_{1}\le C$ and hence it follows
$$  \frac{\|\gamma_n h u_n  e^{b_n u_n^2}\|_1}{\| u_n\|_q}=\frac{\|\gamma_n h m_n u_n  e^{b_n u_n^2}\|_1}{\|m_n u_n\|_q}\to 0 $$
as $n\to +\infty$. Moreover it is easy to check with \eqref{sn2} and \eqref{c_n} that 
$$\|s_n\|_{1}\le C \|u_n\|_q $$
and we have a uniform bound for $-\Delta_g w_n$ in $L^1(\Sigma)$. Therefore $w_n$ is uniformly bounded in  $W^{1,s}(\Sigma)$ for any $1<s<2$ (see \cite{struL1} for a reference on open sets in $\R^2$). The weak limit $w$ of  $w_n$ will satisfy
\[
\int_{\Sigma} \nabla w\cdot \nabla \ph \; dv_g = \lambda \int_{\Sigma} |w|^{q-2} w \ph dv_g
\]
for any $\ph\in C^1(\Sigma)$ such that $\int_{\Sigma} \ph dv_g =0$.
But, since $\lambda<\lambda_q(\Sigma,g)$, this implies $w=0$ which contradicts $\|w_n\|_q=1$. Hence $\|m_n u_n\|_q \le C$.

This implies that $-\Delta_g (m_n u_n )$ is uniformly bounded in $L^1(\Sigma)$ and, as before, $m_n u_n$ is uniformly bounded in $W^{1,s}(\Sigma)$ for any $s\in (1,2)$. By Lemma \ref{mis} we have $m_n u_n\rw G_p^\lambda$ weakly in $W^{1,s}(\Sigma)$, $s\in (1,2)$ and strongly in $L^r$ for any $r\ge 1$.


From Lemma \ref{punto} we get $|\nabla u_n|^2 \rw \delta_p$ and $u_n$ is uniformly bounded in $L^\infty_{loc}(\Sigma\bs\{p\})$. This implies the boundedness of $-\Delta_g (m_n u_n)$ in $L^s_{loc}(\Sigma\bs\{p\})$ for some $s>1$ which gives a uniform bound for $m_n u_n$ in $W^{2,s}_{loc}(\Sigma\bs \{p\})$. Then, by elliptic estimates, we get $m_n u_n \ra G_p^\lambda$ in $H^1_{loc}(\Sigma\bs\{p\}) \cap C^0_{loc}(\Sigma\bs \{p\}).$
\end{proof}

As we did in the proof of Theorem \ref{teo 1}, in the next Proposition we will use an Onofri-type inequality (Corollary \ref{cor CC sing}) to control the energy on a small scale.

\begin{prop}\label{keystep}
We have $\alpha(p)=\ov\alpha$ and for any $L>0$
$$
\limsup_{n\to \infty}\int_{\Omega_{L r_n}(p_n)} h e^{b_n u_n^2} dv_g\le \frac{\pi K(p) e^{1+\ov{\beta} A^\lambda_p}}{1+\ov{\alpha}}.
$$
\end{prop}

\begin{proof}
Let us observe that 
\begin{equation}\label{cclo}
\begin{split}
\int_{D_{L r_n}(x_n)}|x|^{2\alpha(p)}e^{b_n v_n^2}\, dx &=\int_{D_{L r_n}(x_n)}|x|^{2(\alpha(p)-\ov\alpha)+2\ov\alpha}e^{b_n v_n^2}\, dx\\&
\leq \left(Lr_n \right)^{2(\alpha(p)-\ov\alpha)}\int_{D_{L r_n}(x_n)}|x|^{2\ov\alpha}e^{b_n v_n^2}\, dx.
\end{split}
\end{equation}
Fix $\delta>0 $ and set $\tau_n = \int_{\Omega_\delta}|\nabla u_n|^2 dv_g = \int_{D_\delta} |\nabla v_n |^2 dy$. Observe that, by Lemma \ref{conv Gr},
\begin{equation}\label{tau}
m_n^2(1-\tau_n) = \int_{\Sigma\bs \Omega_{\delta}}|\nabla G_p^\lambda|^2 dv_g +o(1),
\end{equation}
and 
\begin{equation}\label{normaq}
m_n^2 \|u_n\|_q^2 =   \|G_p^\lambda\|_{q}^2 +o(1).
\end{equation}
Since by Lemma \ref{Grad Gr} we have 
\begin{equation}\label{norm Gr}
\int_{\Sigma\bs \Omega_\delta} |\nabla G_p^\lambda|^2 dv_g = -\frac{1}{2\pi} \log \delta + O(1) \stackrel{\delta\to 0}{\ra}  +\infty,
\end{equation}
for $\delta$ sufficiently small, we obtain

\begin{equation}\label{tau e norma}
\begin{split}
\tau_n (1+\lambda \|u_n\|_{2}^2)&=\left(1-\frac{1}{m_n^2}\int_{\Sigma\setminus\Omega_\delta}|\nabla G_p^\lambda|^2 dv_g + o \left(\frac{1}{m_n^2}\right) \right)\left(1+\frac{\lambda}{m_n^2}\|G_p^\lambda\|_{q}^2 +o\left(\frac{1}{m_n^2}\right)\right)\\&
= 1-  \frac{1}{m_n^2}\left( \int_{\Sigma\setminus\Omega_\delta} |\nabla G_p^\lambda|^2 dv_g -\lambda \| G_p^\lambda\|_{q}^2\right) + o\left(\frac{1}{m_n^2}\right)<1.
\end{split}
\end{equation}
We denote $d_n:= \sup_{\partial D_{\delta}} v_n$ and $w_n := (v_n -d_n)^+ \in H^1_0(D_\delta)$. Applying Holder's inequality we have 
\begin{equation}\label{Hold}
\begin{split}
&\int_{D_{L_{r_n}}(x_n)} |x|^{2\ov{\alpha}}e^{b_n v_n^2} dx = e^{b_n d_n^2} \int_{D_{L_{r_n}}(x_n)} |x|^{2\ov{\alpha}} e^{b_n w_n^2+2b_nd_n w_n }dx  \\&
\le e^{b_n d_n^2}\left(\int_{D_{L r_n}(x_n)}|x|^{2\ov{\alpha}} e^{\beta_n\frac{w_n^2}{\tau_n}}  dx\right)^{\tau_n (1+\lambda \|u_n\|_q^2)}\left( \int_{D_{Lr_n}(x_n)} |x|^{2\ov{\alpha}}  e^{\frac{2 b_n w_n d_n }{1-\tau_n (1+\lambda \|u_n\|_q^2)}}\right)^{1-\tau_n (1+\lambda \|u_n\|_q^2)}.
\end{split}
\end{equation}
Observe that, for $n\to+\infty$, we have that $\frac{w_n}{\sqrt{\tau_n}}\ra 0$ uniformly on $D_{\delta}\bs D_{\delta'}$ for any $0<\delta'<\delta$.  Thus applying Corollary \ref{mi serve} to the function $\frac{w_n}{\sqrt{\tau_n}}$ with $\delta_n= L r_n$, we find
\begin{equation}\label{CCloc}
\limsup_{n\to \infty}\int_{D_{L r_n}(x_n)} |x|^{2\ov{\alpha}} e^{{\beta_n} \frac{w_n^2}{\tau_n}} dx \le \frac{\pi e}{1+\ov{\alpha}} \delta^{2(1+\ov\alpha)}.
\end{equation}

Using Corollary \ref{cor CC sing} we  find 
\[
\begin{split} 
\int_{D_{Lr_n(x_n)}} |x|^{2\ov{\alpha}}  e^{\frac{2b_n w_n d_n}{1-\tau_n (1+\lambda \|u_n\|_q^2)}}& \le \int_{D_{\delta}} |x|^{2\ov{\alpha}}  e^{\frac{2b_n w_n d_n}{1-\tau_n (1+\lambda \|u_n\|_q^2)}}dx\\& 
\le \frac{\pi e^{1+\frac{4 b_n^2 d_n^2 \tau_n }{16\pi (1+\ov{\alpha}){(1-\tau_n (1+\lambda \|u_n\|_q^2)^2}}}}{1+\ov\alpha}  \delta^{2(1+\ov\alpha)}\\&
\le \frac{\pi e^{1+\frac{ b_n d_n^2 \tau_n (1+\lambda\|u_n\|_q^2) }{(1-\tau_n (1+\lambda\|u_n\|_q^2)^2}}}{1+\ov\alpha}  \delta^{2(1+\ov\alpha)}.
\end{split}
\]
Combining this with \eqref{cclo}, \eqref{Hold} and \eqref{CCloc}, we find 
\begin{equation}\label{qu}
\limsup_{n\to \infty} \int_{D_{L_{r_n}(x_n)}} |x|^{2\alpha(p)}e^{b_n v_n^2} dx \le  \frac{\pi e\delta^{2(1+\ov\alpha)}}{1+\ov\alpha}\limsup_{n\to \infty} \left(Lr_n\right)^{2(\alpha(p)-\ov\alpha)} e^{\frac{ b_n d_n^2 }{1-\tau_n (1+\lambda\|u_n\|_q^2)}}.
\end{equation}
Using   \eqref{tau e norma} and Lemma \ref{conv Gr},
\begin{equation}\label{lim Gr}
\lim_{n\to \infty} \frac{ b_n d_n^2 }{1-\tau_n (1+\lambda\|u_n\|_q^2)} =  \frac{\ov{\beta} (\sup_{\partial B_\delta} G_p^\lambda)^2}{\left( \int_{\Sigma\setminus\Omega_\delta} |\nabla G_p^\lambda|^2 dv_g -\lambda \| G_p^\lambda\|_{q}^2\right) } =: H(\delta).
\end{equation}

Notice that by Lemma \ref{Grad Gr} and \eqref{exp Gr} we find 
\begin{equation}\label{accadelta}
H(\delta) = -2(1+\ov\alpha) \log \delta + \ov{\beta} A^\lambda_p +o_\delta(1).
\end{equation}

From \eqref{qu}, \eqref{lim Gr} we obtain
\begin{equation}\label{dai dai}
\begin{split}
\limsup_{n\to+\infty}\int_{\Omega_{Lr_n}(p_n)}he^{b_n u_n^2}\, dv_g&=\limsup_{n\to \infty} \int_{D_{L_{r_n}}(x_n)} V(x) |x|^{2\alpha(p)}e^{b_n v_n^2} dx \\&
\le  \frac{K(p) \pi e\delta^{2(1+\ov\alpha)}}{1+\ov\alpha} e^{H(\delta)} \limsup_{n\to+\infty}\left(Lr_n\right)^{2(\alpha(p)-\ov\alpha)}.
\end{split}
\end{equation}
If $\alpha(p)>\ov\alpha$ we would have $(Lr_n)^{2(\alpha(p)-\ov\alpha)}\to0$ as $n\to+\infty$. This would imply, using Lemma \ref{supspezzato}, that 
$$ \limsup_{n\to +\infty} E_n(u_n)\leq |\Sigma_{g_h}|,$$
which is a contradiction since $u_n$ is a maximizing sequence. Hence necessary we have $\alpha(p)=\ov\alpha$. Therefore combining \eqref{lim Gr}, \eqref{accadelta} and \eqref{dai dai} we get

\[
\limsup_{n\to+\infty}\int_{\Omega_{Lr_n}(p_n)}he^{b_n u_n^2}\, dv_g
\le  \frac{K(p) \pi e\delta^{2(1+\ov\alpha)}}{1+\ov\alpha} e^{H(\delta)}
=  \frac{K(p)\pi e^{1+\ov{\beta}A^\lambda_p +o_\delta(1)}}{1+\ov\alpha}.
\]

\end{proof}

\begin{proof}[Proof of Proposition \ref{finale}]
The proof follows at once from Lemma \ref{supspezzato} and Proposition \ref{keystep}.
\end{proof}


\section{Test Functions and Existence of Extremals.}\label{sez6}
By Proposition \ref{finale},  in order to prove existence of extremals for  $E^{\ov{\beta},\lambda,q}_{\Sigma,h}$ it   suffices  to show that the value 
$$
\frac{\pi e}{1+\ov\alpha}  \max_{p\in \Sigma, \; \alpha(p)=\ov{\alpha}} K(p)e^{\ov\beta A_p^\lambda} + |\Sigma|_{g_h}.
$$
is exceeded. In this section we will show that this is indeed the case if $\lambda$ is small enough.

\begin{prop}
There exists $\lambda_0>0$ such that $\forall$ $0\le \lambda< \lambda_0$ one has 
$$
\sup_{u\in \mathcal{H}}  E_{\Sigma,h}^{\ov\beta, \lambda q} > \frac{\pi e}{1+\ov\alpha}  \max_{p\in \Sigma,\; \alpha(p)=\ov{\alpha}} K(p)e^{\ov{\beta}A_p^\lambda} + |\Sigma|_{g_h}
$$
\end{prop}
\begin{proof}
Let $p\in\Sigma$ be such that $\alpha(p)=\ov\alpha$ and
\[
K(p)e^{\ov{\beta}A^\lambda_p}=\max_{q\in\Sigma,\; \alpha(q)=\ov\alpha}K(q)e^{\ov{\beta}A^\lambda_q}.
\]
In local coordinates $(\Omega,\psi)$ satisfying \eqref{psi1}-\eqref{psi2} we define
\begin{equation}\label{definizionewepsilon}
w_\eps(x):=\begin{Si}{ll}
c_\eps -\frac{\log\left(1+\left(\frac{|\psi(x)|}{\eps}\right)^{2(1+\ov\alpha)}\right)+L_\eps}{\ov{\beta} c_\eps}\quad\; &  x\in \Omega_{\gamma_\eps \eps} \\
\frac{G_p^\lambda- \eta_\eps \xi }{c_\eps} & x\in \Omega_{2\gamma_\eps \eps}\bs \Omega_{\gamma_\eps \eps} \\\
\frac{G_p^\lambda}{c_\eps} & x\in \Sigma \bs \Omega_{2 \gamma_\eps \eps}
\end{Si}
\end{equation}
and 
$$
u_\eps:= \frac{w_\eps}{\sqrt{1+\frac{\lambda}{c_\eps^2}\|G_p^\lambda\|_q^2 }}
$$
where $c_\eps,L_\eps$ will be chosen later, $\gamma_\eps = |\log\eps|^\frac{1}{1+\ov\alpha}$, $\xi$ is defined as in \eqref{exp Gr} and $\eta_\eps$ is a cut-off function such that $\eta_\eps \equiv 1$ in $\Omega_{\gamma_\eps \eps}$, $\eta_\eps \in C^\infty_0(\Omega_{2\gamma_\eps \eps})$ and $\|\nabla \eta_\eps \|_{L^\infty(\Sigma)}= O(\frac{1}{\gamma_\eps \eps})$. In order to have $u_\eps \in H^1(\Sigma)$ we choose $L_\eps$ so that 
\begin{equation}\label{conti}
\ov{\beta} c^2_\eps-L_\eps =\log\left(\frac{1+\gamma_\eps^{2(1+\ov\alpha)}}{\gamma_\eps^{2(1+\ov\alpha)}}\right) +\ov{\beta} A_p^\lambda-2(1+\ov\alpha)\log \eps.
\end{equation}
Observe that
\begin{equation}
\int_{\Omega_{\gamma_\eps \eps}}|\nabla w_\eps|^2 dv_g = \frac{1}{\ov{\beta}c_\eps^2}\left(\log(1+\gamma_\eps^{2(1+\ov\alpha)})-1 +O(|\log \eps|^{-2})\right).
\end{equation}
Since $\xi\in C^1(\ov{D_{\delta_0}})$ and $\xi(x)= O(|x|)$ we have
\[
\begin{split}
&\int_{\Omega_{2\gamma_\eps \eps}\bs \Omega_{\gamma_\eps \eps}} |\nabla (\eta_\eps \xi )|^2 \, dv_g =  \int_{\Omega_{2\gamma_\eps \eps}\bs \Omega_{\gamma_\eps \eps}} |\nabla \eta_\eps|^2 \xi^2\, dv_g \\&
+ \int_{\Omega_{2\gamma_\eps \eps}\bs \Omega_{\gamma_\eps \eps}} |\nabla \xi|^2 \eta_\eps ^2\, dv_g + 2\int_{\Omega_{2\gamma_\eps \eps}\bs \Omega_{\gamma_\eps \eps}} \eta_\eps \xi  \nabla \eta_\eps \cdot\nabla \xi\, dv_g\\&
= O((\gamma_\eps \eps)^2),
\end{split}
\]
and similarly
$$
\int_{\Omega_{2\gamma_\eps \eps}\bs \Omega_{\gamma_\eps \eps}} \nabla G_p^\lambda \cdot \nabla (\eta_\eps \xi ) dv_g =O(\gamma_\eps \eps ).
$$
By Lemma \ref{Grad Gr} we have
$$
c_\eps^2\int_{\Sigma\bs \Omega_{\gamma_\eps \eps}}|\nabla w_\eps|^2 dv_g = \int_{\Sigma\bs \Omega_{\gamma_\eps \eps}}|\nabla G_p^\lambda|^2 +O(\gamma_\eps \eps) = $$
$$
= -\frac{1}{2\pi} \log{\gamma_\eps \eps } +A_p^\lambda +\lambda\|G^\lambda_p\|^2_q+O(\gamma_\eps \eps|\log(\gamma_\eps \eps)|).
$$
Observe that $\gamma_\eps \eps \log(\gamma_\eps \eps)=o(|\log \eps|^{-2})$,
therefore we get 
$$
\int_{\Sigma} |\nabla w_\eps|^2 dv_g  =  \frac{1}{\ov{\beta} c_\eps^2 } \left(   -1-2(1+\ov\alpha)\log \eps + \ov{\beta }A_p^\lambda +\ov{\beta}\lambda\|G_p^\lambda\|_q^2 +O(|\log \eps|^{-2})\right).
$$
If we chose $c_\eps$ so that
\begin{equation}\label{conti2}
\ov{\beta} c_\eps^2 =  -1-2(1+\ov\alpha)\log \eps + \ov{\beta }A_p^\lambda  +O(|\log \eps|^{-2}),
\end{equation}
then $u_\eps - \ov{u}_\eps \in  \mathcal{H}$. Observe also that  \eqref{conti}, \eqref{conti2} yield
\begin{equation}\label{conti 3}
L_\eps= -1 + O(|\log\eps|^{-2}).
\end{equation}
and 
\begin{equation}\label{conti 4}
2\pi c_\eps^2 = |\log \eps|+O(1).
\end{equation}
Since $0\le w_\eps \le O(c_\eps) $ in $\Omega_{\gamma_\eps \eps}$ we get
$$
\int_{\Omega_{\gamma_\eps \eps}} w_\eps  dv_g = O(c_\eps (\gamma_\eps \eps)^2) = o(|\log \eps|^{-2}).
$$ 
Moreover 
\[
\begin{split}
\int_{\Sigma\bs \Omega_{\gamma_\eps \eps}} w_\eps\,  dv_g &=   \int_{\Sigma\bs \Omega_{\gamma_\eps \eps}} \frac{G_p^\lambda}{c_\eps}\,  dv_g - \int_{\Omega_{2\gamma_\eps \eps} \bs \Omega_{\gamma_\eps \eps}} \frac{\eta_\eps \xi }{c_\eps } \, dv_g \\&
= O\left(\frac{ (\gamma_\eps \eps)^2 |\log(\gamma_\eps \eps)|}{c_\eps}\right) +O\left(\frac{(\gamma_\eps \eps)^3}{c_\eps}\right)\\&
 = o(|\log\eps|^{-2})
\end{split}
\]
therefore 
\begin{equation}\label{omegabar}
\ov{w}_\eps =o(|\log \eps|^{-2})= o(c_\eps^{-4}).
\end{equation}
From \eqref{conti2}, \eqref{conti 3} and \eqref{omegabar} it follows that in $\Omega_{\gamma_\eps \eps}$ we have
$$
\ov{\beta}(w_\eps- \ov{w}_\eps)^2 \ge  \ov\beta c_\eps^2 - 2 L_\eps - 2 \log\left(1+\left(\frac{|\psi(x)|}{\eps}\right)^{2(1+\ov\alpha)}\right)  +o(c_\eps^{-2}).
$$
We have
\[
c_\eps^2 \|w_\eps-\ov{w}_\eps \|^2_q  \ge \left( \int_{\Sigma\bs \Omega_{2\gamma_\eps \eps} }|G_p^\lambda -c_\eps\ov{w}_\eps|^q \, dv_g \right)^\frac{2}{q} \geq \|G_p^\lambda\|_q^2 +o(c_\eps^{-2}),
\]
where the last inequality follows from \eqref{definizionewepsilon} and Bernoulli's inequality after splitting the integral on regions where $|G_p^\lambda|\geq |c_\eps \ov{w}_\eps|$ and $|G_p^\lambda|\leq |c_\eps \ov{w}_\eps|$.
Therefore we find 
\begin{equation}\label{aiutomaggiore}
\begin{split}
&\frac{1}{1+\frac{\lambda}{c_\eps^2}\|G_p^\lambda\|_q^2 }  \left( 1 + \frac{\lambda\|w_\eps-\ov w_\eps \|_q^2}{1+\frac{\lambda}{c_\eps^2}\|G_p^\lambda\|_q^2}\right)\ge \frac{1+2\frac{\lambda}{c_\eps^2}\|G_p^\lambda\|_q^2+o(c_\eps^{-4})}{\left(1+\frac{\lambda}{c_\eps^2}\|G_p^\lambda\|_q^2\right)^2} \\&
= 1- \frac{\lambda^2 \|G_p^\lambda\|_q^4}{c_\eps^4} +o(c_\eps^{-4}).
\end{split}
\end{equation}
Therefore 
$$
\ov{\beta}(u_\eps- \ov{u}_\eps)^2 (1+\lambda\|u_\eps-\ov u_\eps\|_q^2)\ge   \ov\beta c_\eps^2 - 2 L_\eps - 2 \log\left(1+\left(\frac{|\psi(x)|}{\eps}\right)^{2(1+\ov\alpha)}\right) - \frac{\ov{\beta}\lambda^2 \|G_p^\lambda\|_q^4}{c_\eps^2} +o(c_\eps^{-2}).
$$

It follows that
\[
\begin{split}
&\int_{\Omega_{\gamma_\eps \eps}} h e^{\ov{\beta} (u_\eps-\ov{u}_\eps )^2(1+\lambda\|u_\eps-\ov{u}_\eps\|^2_q)} dv_g\\& 
\ge \int_{D_{\gamma_\eps \eps }}|x|^{2\ov\alpha} (V(0)+O(\gamma_\eps \eps ))\frac{ e^{\ov{\beta} c_\eps^2-2 L_\eps - \frac{\ov{\beta}\lambda^2 \|G_p^\lambda\|_q^4}{c_\eps^2} +o(c_\eps^{-2})}}{\left(1+\left(\frac{|x|}{\eps}\right)^{2(1+\ov\alpha)}
\right)^2} dx\\&
=\frac{\pi V(0) \eps^{2(1+\ov\alpha)} \gamma_\eps^{2(1+\ov\alpha)}}{(1+\ov\alpha)(1+\gamma_\eps^{2(1+\ov\alpha)})} e^{\ov{\beta} c_\eps^2 -2 L_\eps- \frac{\ov{\beta}\lambda^2 \|G_p^\lambda\|_q^4}{c_\eps^2} +o(c_\eps^{-2})} (1+O(\gamma_\eps \eps))\\&
= \frac{\pi V(0) \eps^{2(1+\ov\alpha)}}{(1+\ov\alpha)} e^{\ov{\beta} c_\eps^2 -2 L_\eps- \frac{\ov{\beta}\lambda^2 \|G_p^\lambda\|_q^4}{c_\eps^2} +o(c_\eps^{-2})} (1+O(c_\eps^{-4})).
\end{split}
\]
Using \eqref{conti2} and  \eqref{conti 3} we find 
\[
\ov{\beta}c_\eps^2 -2L_\eps=-2(1+\ov\alpha) \log\eps + 1+\ov{\beta} A_p^\lambda +O(c_\eps^{-4})
\]
so that 
\begin{equation}\label{omegaepseps}
\int_{\Omega_{\gamma_\eps \eps}} h e^{\ov{\beta} (u_\eps-\ov{u}_\eps )^2(1+\lambda\|u_\eps-\ov{u}_\eps\|^2_q)} =  \frac{\pi V(0) e^{1+\ov{\beta}A_p^\lambda }}{(1+\ov\alpha)} \left(1-\frac{\ov{\beta}\lambda^2 \|G_p^\lambda\|_q^4}{c_\eps^2} +o(c_\eps^{-2})  \right).
\end{equation}
Finally, with \eqref{omegabar} and \eqref{aiutomaggiore}, we observe that 
\begin{equation}\label{omegaepseps2}
\begin{split}
&\int_{\Sigma\bs \Omega_{2\gamma_\eps \eps}}  h e^{\ov{\beta} (u_\eps-\ov{u}_\eps)^2(1+\lambda\| u_\eps-\ov{u}_\eps\|_q^2)} dv_g \\&
\ge \int_{\Sigma\bs \Omega_{2\gamma_\eps \eps}} h \, dv_g +  \ov{\beta}(1+\lambda\| u_\eps-\ov{u}_\eps\|_q^2)\int_{\Sigma\bs \Omega_{2\gamma_\eps \eps }} h (u_\eps-\ov{u}_\eps)^2 \, dv_g\\&
\geq |\Sigma|_{g_h} +O((\gamma_\eps \eps)^{2(1+\ov\alpha)}) +\ov{\beta}\left(1-\frac{\lambda^2\| G^\lambda_p\|^4_q}{c_\eps^4}+o(c_\eps^{-4})\right)\int_{\Sigma\bs \Omega_{2\gamma_\eps \eps}} h (w_\eps-\ov{w}_\eps)^2 dv_g\\&
= |\Sigma|_{g_h} + \frac{\ov{\beta}\|G_p^\lambda\|_{L^2(\Sigma,g_h)}}{c_\eps^2} +o(c_\eps^{-2}).
\end{split}
\end{equation}
Hence from \eqref{omegaepseps} and \eqref{omegaepseps2} it follows that
\[
E^{\ov{\beta}, \lambda, q}_{\Sigma, h}(u_\eps-\ov{u}_\eps) \ge \frac{\pi K(p)}{1+\ov\alpha}e^{1+\ov{\beta} A_p^\lambda} + |\Sigma|_{g_h} +\frac{\ov{\beta}}{c_\eps^2} \left(\|G_p^\lambda\|_{L^2(\Sigma,g_h)} - \frac{\pi K(p)e^{1+\ov{\beta}A^\lambda_p}\lambda^2 \|G_p^\lambda\|_q^4}{1+\ov\alpha}\right) +o(c_\eps^{-2}),
\]
where we used that by definition $K(p)=V(0)$.
By Lemma \ref{as0}, we know that 
$$
\left(\|G_p^\lambda\|_{L^2(\Sigma,g_h)} - \frac{\pi K(p)e^{1+\ov{\beta}A^\lambda_p}\lambda^2 \|G_p^\lambda\|_q^4}{1+\ov\alpha}\right) \ra \|G_p^0\|_{L^2(\Sigma,g_h)}>0
$$
as $\lambda\to 0$ thus for sufficiently small $\lambda$ we get the conclusion. 
\end{proof}

To finish the proof of Theorem \ref{teo 3} we have to treat the case $\lambda>\lambda_{q}(\Sigma,g)$.   Will use a family of test functions similar to the one used in \cite{LuYang}. 
\begin{lemma}\label{infinito}
If $\beta > \ov{\beta}$ or $\beta= \ov{\beta}$ and $\lambda>\lambda_q(\Sigma,g)$, we have 
$$
\sup_{\mathcal{H}} E^{\ov\beta,\lambda,q}_{\Sigma,h} = +\infty. 
$$
\end{lemma}
\begin{proof}
Take $p\in \Sigma$  such that $\alpha(p)=\ov\alpha$ and a local chart $(\Omega,\psi)$  satisfying  \eqref{psi1}-\eqref{psi2}. Let us define  $v_\eps :D_{\delta_0} \ra [0,+\infty)$,
$$
v_\eps(x):=\frac{1}{\sqrt{2\pi}}\begin{Si}{cl}
\sqrt{\log{\frac{\delta_0}{\eps}}} & \; |x|\le \eps\\
\frac{\log\frac{\delta_0}{|x|}}{\sqrt{\log{\frac{\delta_0}{\eps}}} } & \;\eps \le |x|\le \delta_0\\
\end{Si}
$$
and 
$$
u_\eps (x):= \begin{Si}{cl}
v_\eps(\psi (x)) & \; x\in \Omega\\
0 & \; x\in \Sigma \bs \Omega.
\end{Si}
$$
It is simple to verify that
$$
\int_{\Sigma}|\nabla u_\eps|^2 dv_g =  \int_{D_{\delta_0}} |\nabla v_\eps|^2 dx = 1,
$$
thus $u_\eps -\ov{u}_\eps \in \mathcal{H}$.
By direct computation one has 
\begin{equation}\label{mediaueps}
\ov{u}_\eps=O\left( \left(\log \frac{1}{\eps}\right)^{-\frac{1}{2}}\right),
\end{equation} 
hence in $\Omega_\eps$ we have 
$$
(u_\eps -\ov{u}_\eps)^2 =  \frac{1}{2\pi}\log \left(\frac{\delta_0}{\eps}\right) + O(1).
$$
Thus if $\beta>\ov{\beta}$ we have
$$
E_{\Sigma,h}^{\beta,\lambda,q} (u_\eps -\ov{u}_\eps)  \ge E_{\Sigma,h}^{\beta,0,q} (u_\eps -\ov{u}_\eps ) \ge \int_{\Omega_\eps } h e^{\beta(u_\eps -\ov{u}_\eps )^2} dv_g \ge   \frac{C}{\eps^{\frac{\beta}{2\pi}}} \int_{D_\eps} |x|^{2\ov{\alpha}} dx=$$
$$
= \frac{C \pi }{1+\ov\alpha} \eps^{2(1+\ov{\alpha})-\frac{\beta}{2\pi}} 
= \tilde{C} \eps^{\frac{\ov{\beta}-\beta}{2\pi}} \ra +\infty
$$
as $\eps\ra 0$.
For the case $\beta= \ov{\beta}$ and $\lambda> \lambda_q(\Sigma, g)$ we take a function $u_0\in H^1(\Sigma)$ such that 
\begin{equation}\label{autofun}
\begin{Si}{c}
\|\nabla u_0\|_2^2 = \lambda_q (\Sigma,g)\|u_0\|_q^2
\vspace{0.3cm}\\
\int_{\Sigma} u_0 \; dv_g = 0\vspace{0.3cm}\\
\|u_0\|_q^2 =1.
\end{Si}
\end{equation}
This function $u_0$ will also satisfy
\[
-\Delta_g u_0 = \lambda_q \|u_0\|_{q}^{2-q}|u_0|^{q-2} u_0 -c\; 
\]
where 
$$
c = \frac{\lambda_q}{|\Sigma|}\|u_0\|_q^{2-q} \int_{\Sigma}|u_0|^{q-2} u_0\; dv_g.
$$
Let us take $t_\eps, r_\eps \ra 0 $ such that
\begin{equation}\label{re}
t_\eps^2 |\log\eps| \ra +\infty, \qquad
\frac{r_\eps}{\eps}, \ra+\infty \qquad \mbox{and} \qquad \frac{\log ^2r_\eps}{t_\eps^2|\log\eps|} \ra 0. 
\end{equation}

We define
\begin{equation}\label{defweps}
w_\eps:=  u_\eps \eta_\eps + t_\eps u_0
\end{equation}
where $\eta_\eps \in C^\infty_0(\Omega_{2r_\eps})$ is a cut-off function such that $\eta_\eps \equiv 1 $ in $\Omega_{r_\eps}$, $0\le \eta_\eps \le 1$ and   $|\nabla \eta_\eps |=O(r_\eps^{-1})$. 
It is straightforward that 
\begin{equation}\label{mediaweps}
\ov w_\eps=O(|\log\eps|^{-\frac 12}).
\end{equation}
Observe that
$$
\|\nabla w_\eps\|_2^2 = \int_{\Sigma}|\nabla (u_\eps\eta_\eps)|^2  dv_g + t_\eps^2 \|\nabla u_0\|_2^2 + 2 t_\eps \int_{\Sigma} \nabla u_0 \cdot \nabla  (u_\eps \eta_\eps ) dv_g.
$$ 
Using the definition of $u_\eps, \eta_\eps$ and \eqref{re} we find
$$
\int_{\Sigma}|\nabla \eta_\eps|^2 u_\eps^2 dv_g = O(r_\eps^{-2})\int_{\Omega_{2r_\eps}\bs \Omega_{r_\eps}}  u_\eps^2 dv_g = O\left(|\log \eps|^{-1}\log^2 r_\eps \right)= o(t_\eps^2)
$$
and 
$$
\left|\int_{\Sigma} u_\eps \eta_\eps \nabla u_\eps \cdot \nabla \eta_\eps  dv_g \right| \le O(r_\eps^{-1}) \int_{\Omega_{2 r_\eps} \bs \Omega r_\eps} |\nabla u_\eps| u_\eps dv_g =  O(|\log r_\eps| |\log \eps|^{-1}) =o(t_\eps^2).
$$
Thus
$$
\|\nabla (u_\eps \eta_\eps)\|_2^2 = \int_{\Sigma}|\nabla u_\eps|^2 \eta_\eps^2  dv_g + o(t_\eps^2) \le 1+o(t_\eps^2).
$$
Moreover \eqref{autofun} gives $\|\nabla u_0\|_2^2 = \lambda_q$ and 
$$
\left|\int_{\Sigma} \nabla u_0 \cdot \nabla  (u_\eps \eta_\eps ) dv_g\right| =\left| \lambda_q  \int_{\Sigma} (|u_0|^{q-2} u_0-c) \eta_\eps u_\eps dv_g \right|= O(1) \int_{\Sigma} u_\eps dv_g = O(|\log \eps|^{-\frac{1}{2}}) =o(t_\eps).
$$
Hence we have 
$$
\|\nabla w_\eps\|^2_2 \le 1 + \lambda_q t_\eps^2 +o(t_\eps^2). 
$$
Furthermore by dominated convergence we have,
$$
\|w_\eps-\ov{w}_\eps \|_{q}^2\ge t_\eps^2 \left(\int_{\Sigma\bs \Omega_{2 r_\eps}} |u_0-\frac{\ov{w}_\eps}{t_\eps}|^q dv_g\right)^\frac{2}{q} = t_\eps^2 \| u_0\|_q^2 +o(t_\eps^2)= t_\eps^2 +o(t_\eps^2),
$$
thus we get 
$$
\frac{1}{\|\nabla w_\eps\|^2_2} \left(1+\lambda \frac{\|w_\eps-\ov w_\eps\|_q^2}{\|\nabla w_\eps\|_2^2}\right) \ge  1 + (\lambda-\lambda_q)t_\eps^2  +o(t_\eps^2).
$$
Finally, using \eqref{mediaweps}, in $\Omega_\eps$ we find 
\[
\begin{split}
&\frac{4\pi (1+\ov{\alpha})(w_\eps -\ov{w}_\eps)^2}{\|\nabla w_\eps\|^2_2} \left(1+\lambda \frac{\|w_\eps-\ov w_\eps\|_q^2}{\|\nabla w_\eps\|_2^2}\right) \\&
= \left(2(1+\ov{\alpha})|\log\eps|+O(1)\right)  \left(  1 + (\lambda-\lambda_q)t_\eps^2  +o(t_\eps^2)\right)\\&
=  -2(1+\ov{\alpha}) \log \eps + (\lambda -\lambda_q) t_\eps^2 |\log \eps | + O(1),
\end{split}
\]
so that
$$
E_{\Sigma,h}^{\ov\beta, \lambda, q} \left(\frac{w_\eps-\ov{w}_\eps}{\|\nabla w_\eps\|_2}\right) \ge \int_{\Omega_\eps} he^{\frac{4\pi (1+\ov{\alpha})(w_\eps -\ov{w}_\eps)^2}{\|\nabla w_\eps\|^2_2} \left(1+\lambda \frac{\|w_\eps\|_q^2}{\|\nabla w_\eps\|_2^2}\right)} dv_g \ge $$
$$
\ge  c \eps^{-2(1+\ov{\alpha})} e^{ (\lambda-\lambda_q) t_\eps^2|\log \eps|+O(1)} \int_{D_\eps}|y|^{2 \ov{\alpha}} dy = \tilde{c} e^{ (\lambda-\lambda_q) t_\eps^2|\log \eps|} \ra +\infty
$$
as $\eps \to0$. 
\end{proof}

\begin{oss}\label{caso limite}
If there exists a point $p\in \Sigma$ such that $\alpha(p)=\ov\alpha$ and $u_0(p)\neq 0$, then one can argue as in \cite{LuYang} to prove that,
$$
\sup_{\mathcal{H}} E^{\ov{\beta},\lambda,q}_{\Sigma,h} = +\infty
$$
also for $\lambda=\lambda_q(\Sigma,g_0)$. This is always true if $\ov{\alpha}=0$.
\end{oss}


\appendix
\section{Onofri-type Inequalities for Disks.}\label{sez2}

Let $(\Sigma,g)$ be a smooth, closed Riemannian surface. As a consequence of \eqref{MT3} one gets 
\begin{equation}\label{MT sup}
 \log\left(\frac{1}{|\Sigma|} \int_{\Sigma} e^{u-\ov{u}} dv_{g} \right) \le \frac{1}{16\pi} \int_{\Sigma}|\nabla_g u|^2 dv_g + C(\Sigma,g).
\end{equation}
While it is well known that the coefficient $\frac{1}{16\pi}$ is sharp, the optimal value of $C(\Sigma,g)$ is harder to determine. For the special case of the standard Euclidean sphere $(S^2,g_0)$, Onofri  (\cite{onofri}) proved that $C(S^2,g_0)=0$ and gave a complete characterization of the extremal functions for \eqref{MT sup}.
\begin{prop}[\cite{onofri}]\label{Onofri} $\forall\; u\in H^1(S^2)$ we have
$$\log \left(\frac{1}{4\pi} \int_{S^2} e^{u-\ov{u}} dv_{g_0} \right) \le \frac{1}{16\pi }\int_{S_2}|\nabla_{g_0} u|^2 dv_{g_0},$$
with equality holding if and only if $e^{u}g_0$ is a metric on $S^2$ with positive constant Gaussian curvature, or, equivalently, $u=\log |\det d\ph| +c$ with $c\in \R$ and  $\ph:S^2 \ra  S^2$ a conformal diffeomorphism of $S^2$.
\end{prop}

We will prove now Proposition \ref{OnofrisuD} 
by means of the stereographic projection.

\begin{proof}
Let us fix Euclidean coordinates  $(x_1,x_2,x_3)$ on $S^2\subseteq \R^3$ and denote $N:=(0,0,1)$ and $S=(0,0,-1)$ the north and the south pole. Let us consider the stereographic projection $\pi:S^{2}\bs\{N\}\ra\R^2$
$$
\pi(x):=\left(\frac{x_1}{1-x_3},\frac{x_2}{1-x_3}\right).
$$
It is well known that $\pi$ is a conformal diffeomorphism and 
\begin{equation}\label{pullbb}
 \left(\pi^{-1}\right)^* g_0 = e^{u_0} |dx|^2
\end{equation}
where
\begin{equation}\label{u0}
u_0(x) = \log\left( \frac{4}{\left(1+|x|^2\right)^2}\right)\end{equation}
satisfies
\begin{equation}\label{eq u0}
-\Delta u_0 = 2 e^{u_0}\qquad \mbox{ on }\R^2.
\end{equation}

Given $r>0$, let $D_r:=\{x\in \R^2\;:\; |x|<r\}$ be the disk of radius $r$ and $S_r^2 = \pi^{-1}(D_r)$. We consider the map $T_r:H^1_0(D_r)\ra H^1(S^2)$ defined by
$$
T_r u(x):= \begin{Si}{cc}
u( \pi(x))-u_0(\pi(x)) & \mbox{ on } S^2_r \\
-2\log\left(\frac{2}{1+r^2}\right) & \mbox{ on } S^2\bs S^2_r.
\end{Si}
$$ 
Using \eqref{pullbb}  we find
\begin{equation}\label{exp}
\int_{S^2}  e^{T_r u} dv_{g_0} \ge \int_{S^2_r}  e^{T_r u} dv_{g_0}= \int_{D_r}  e^{T_r u ( \pi^{-1}(y))} e^{u_0} dy= \int_{D_r} e^{u(y)} dy.
\end{equation}
Moreover, by \eqref{eq u0},
\[
\begin{split}
\int_{S^2_r} |\nabla T_r u|^2dv_{g_0}& = \int_{D_r}|\nabla u|^2 dx-2\int_{D_r} \nabla u_0 \cdot \nabla u\; dy +\int_{D_r} |\nabla u_0|^2 dy  \\&
=  \int_{D_r}|\nabla u|^2 dy- 4 \int_{D_r} u e^{u_0}dy +\int_{D_r} |\nabla u_0|^2dy \\&
= \int_{D_r}|\nabla u|^2 dy -4 \int_{S^2_r} T_r u \; dv_{g_0} + \left(  \int_{D_r} |\nabla u_0|^2 dy -4\int_{D_r} u_0 e^{u_0} dy \right).  
\end{split}
\]
It is easy to check with a direct computation that one has 
$$
\int_{D_r} |\nabla u_0|^2 dy = 16\pi\left( \log(1+r^2)-\frac{r^2}{1+r^2} \right)
$$
and
$$
\int_{D_r} u_0 e^{u_0} dy= 8\pi \log 2- 8\pi+o(1), 
$$
where $o(1) \ra 0$ as $r\to +\infty$. 
Thus we get
\begin{equation}\label{grad}
\begin{split}
&\int_{S^2}|\nabla T_ru|^2\, dv_{g_0} +4 \int_{S^2} T_ru \, dv_{g_0} \\&
=\int_{D_r}|\nabla u|^2 dy+  16\pi \left( \log(1+r^2)  +1 -2 \log 2 +o(1)\right).
\end{split}
\end{equation}
Using \eqref{exp}, \eqref{grad} and Proposition \ref{Onofri} we can conclude 
\begin{equation}\label{Dr}
\begin{split}
\log\left( \frac{1}{\pi} \int_{D_r} e^{u} dy \right)&\le  \log \left(\frac{1}{\pi}  \int_{S^2}  e^{T_r u} dv_{g_0}  \right) \\&
\le  \frac{1}{16\pi} \left( \int_{S^2} |\nabla Tu|^2 dv_{g_0} + 2 \int_{S^2} Tu\,  dv_{g_0} \right)  + 2\log 2 \\&
\le \frac{1}{16 \pi}\int_{D_r} |\nabla u|^2 dy +   \log(1+r^2)+1 +o(1).
\end{split}
\end{equation}
Now if $u\in H^1_0(D)$ we can apply \eqref{Dr} to $u_r(y) = u(\frac{y}{r})$ and since 
$$
\int_{D} e^{u} dx = \frac{1}{r^2} \int_{D_r}e^{u_r(y)}  dy \qquad \mbox{and} \qquad\int_D |\nabla u|^2 dx = \int_{D_r}|\nabla u_r|^2 dy,
$$ we find
$$
\log\left(\frac{1}{\pi}\int_{D}  e^{u} dx \right)\le \frac{1}{16\pi}\int_{D} |\nabla u|^2 dx + 1 +o(1).
$$
As $r\to \infty$ we get the conclusion. 
\end{proof}

As in \cite{adim}, starting from \eqref{Onofri Disco} we can use a simple change of variables to obtain singular Onofri-type inequalities for the unit disk.

\begin{prop}\label{prop2prop3}
Let $-1<\alpha\leq 0$.  Then for any $u\in H^1_0(D)$ we have
\begin{equation}\label{eqprop3}
\log\left( \frac{1+\alpha}{\pi} \int_{D} |x|^{2\alpha} e^{u} dx\right) \le \frac{1}{16\pi(1+\alpha)}\int_{D}|\nabla u|^2 \, dx +  1.
\end{equation}
Moreover, if we restrict ourselves to the space $H^1_{0,rad}(D)$, we have that \eqref{eqprop3} holds true for any $\alpha\in (-1,+\infty]$. 
\end{prop}

\begin{proof}
As we did in the proof of Proposition \ref{CC radial}, for $u\in H^1_{0,rad}(D)$ we consider the function $v(x)=u(|x|^{\frac{1}{1+\alpha}})$, which is again in $H^1_{0,rad}(D)$. The second claim follows at once applying \eqref{Onofri Disco} to $v$. 
As for the first claim, if $\alpha\le 0$ we can use again symmetric rearrangements as we did in the proof of Theorem \ref{teo 1} to remove the symmetry assumption.


\end{proof}

%

Since 
$$
\int_{D} |x|^{2\alpha} dx = \frac{\pi}{1+\alpha},
$$ Proposition \ref{prop2prop3} can be written in a simpler form in terms of the singular metric $g_\alpha=|x|^{2\alpha} |dx|^2$. 
\begin{cor} If $u\in H^{1}_0(D)$ and $-1< \alpha\le 0$ (or $\alpha>0$ and $u\in H^1_{0,rad}(D)$), we have 
$$
\log\left( \frac{1}{|D|_\alpha} \int_{D} e^{u} dv_{g_\alpha}\right) \le \frac{1}{16\pi(1+\alpha)}\int_{D}|\nabla u|^2 dv_{g_\alpha} +  1 
$$
where $|D|_{\alpha}= \frac{\pi}{(1+\alpha)}$ is the measure of $D$ with respect to $g_\alpha$. 
\end{cor}

We stress that the constant 1 appearing in Proposition \ref{prop2prop3} is sharp.   
\begin{prop}\label{inf}
For any $-1<\alpha\leq 0$ we have
$$
\inf_{u\in H^1_{0}(D)} \frac{1}{16\pi (1+\alpha)} \int_D |\nabla u|^2 dx -\log\left(\frac{1}{|D|_\alpha}\int_{D} |x|^{2\alpha}e^{u} dx \right) =-1.
$$
Moreover, if we restrict ourselves to the space $H^1_{0,rad}(D)$, the conclusion above holds true for any $\alpha\in (-1,+\infty)$. 
\end{prop}
\begin{proof}
Let us denote 
$$E_\alpha(u):=\frac{1}{16\pi (1+\alpha)} \int_D |\nabla u|^2 dx -\log\left(\frac{1}{|D|_\alpha}\int_{D} |x|^{2\alpha}e^{u} dv_g \right). 
$$ 
It is sufficient to exhibit a family of functions $u_\eps\in H^1_{0,rad}(D)$ such that $E_\alpha (u_\eps) \stackrel{\eps \to 0}{\ra}-1$. Take  $\gamma_\eps \stackrel{\eps\to 0}{\ra} +\infty$ such that $\eps \gamma_\eps \stackrel{\eps\to 0}{\ra}0$, and define 
$$
u_\eps (x)= \begin{Si}{cl}
-2\log\left(1+\left(\frac{|x|}{\eps}\right)^{2(1+\alpha)}\right) + L_\eps & \mbox{ for }|x|\le \gamma_\eps \eps\\
-4(1+\alpha)\log|x| & \mbox{ for } \gamma_\eps \eps \le |x|\le 1
\end{Si}
$$
where the quantity 
$$\dis{L_\eps}:= 2\log \left( \frac{1+\gamma_\eps^{2(1+\alpha)}}{\gamma_\eps^{2(1+\alpha)}}\right)-4(1+\alpha)\log \eps$$ 
is chosen so that $u_\eps \in H^1_0(D)$. Simple computations show that
\[
\frac{1}{16\pi(1+\alpha)}\int_{D} |\nabla u_\eps|^2 dx
=-1-2(1+\alpha)\log \eps +o_\eps(1)
\]
and 
\[
\begin{split}
\int_{D} |x|^{2\alpha} e^{u_\eps} dx
&=\frac{\eps^{2(1+\alpha)} \gamma_\eps^{2(1+\alpha)}e^{L_\eps} \pi}{(1 + \alpha) (1 + \gamma_\eps^{2 (1 + \alpha)})} +\frac{\pi}{1+\alpha} \left(\frac{1}{(\gamma_\eps \eps)^{2(1+\alpha)}}-1\right)\\&
= \frac{\pi \eps^{-2(1+\alpha)}}{1+\alpha} (1+o_\eps(1)).
\end{split}
\]
Thus 
$$
E_\alpha(u_\eps) \ra -1. 
$$
\end{proof}

To conclude we remark that Propositions \ref{prop2prop3} and \ref{inf} can also be deduced directly using the singular versions of Proposition \ref{Onofri} proved in \cite{mio2}, \cite{mio4}.

\bibliographystyle{plain}
\bibliography{bib}

\end{document}